\documentclass[10pt]{amsart}

\usepackage{amssymb}
\usepackage{graphicx}

\usepackage[usenames,dvipsnames]{color}
\usepackage[hypertex]{hyperref}
\usepackage{soul}

\def\N{{\mathbb{N}}}
\def\Z{{\mathbb{Z}}}
\def\R{{\mathbb{R}}}
\newcommand{\ntn}{\hbar_n}

\newcommand{\dist}{\operatorname{dist}}
\newcommand{\base}{\varepsilon}
\newcommand{\PP}{\mathcal P}
\newcommand{\ssm}{\smallsetminus}

\newcommand{\bnd}{\operatorname{bnd}}

\newcommand{\card}{\operatorname{card}}
\newcommand{\Ends}{\operatorname{Ends}}
\newcommand{\Cuts}{\operatorname{Cuts}}
\newcommand{\Ed}{\operatorname{Ed}}

\newcommand{\overbar}[1]{\mkern 1.5mu\overline{\mkern-1.5mu#1\mkern-1.5mu}\mkern 1.5mu}

\title{Boundaries of $\Z^n$-free groups}

\author{Andrei Malyutin}
\address{St. Petersburg Department of V.A.Steklov Mathematical Institute, Fontanka 27, 191023, 
St.Petersburg, Russia}
\email{andreymalyutin@gmail.com}

\author{Tatiana Nagnibeda}
\address{Section de math${\rm \acute{e}}$matiques, Universit${\rm \acute{e}}$ de Gen${\rm \grave{e}}$ve,
2-4 rue du Li${\rm \grave{e}}$vre, Case postale 64, 1211 Gen${\rm \grave{e}}$ve 4, Suisse}
\email{nagnibeda@gmail.com}

\author{Denis Serbin}
\address{Department of Mathematical Sciences, Stevens Institute of Technology,
1 Castle Point on Hudson, Hoboken, NJ 07030, USA} 
\email{d.e.serbin@gmail.com}



\date{}

\newtheorem{corollary}{Corollary}
\newtheorem{prop}{Proposition}
\newtheorem{theorem}{Theorem}
\newtheorem{lemma}{Lemma}
\newtheorem{defn}{Definition}
\newtheorem{remark}{Remark}

\newtheorem{asser}{Assertion}

\begin{document}

\begin{abstract}
In this paper we study random walks on a finitely generated group $G$ which has a free action on a $\Z^n$-tree. We show that if $G$ is non-abelian and acts minimally, freely and without inversions on a locally finite $\Z^n$-tree $\Gamma$ with the set of open ends $\Ends(\Gamma)$, then for every non-degenerate probability measure $\mu$ on $G$ there exists a unique $\mu$-stationary probability measure $\nu_\mu$ on $\Ends(\Gamma)$, and the space $(\Ends(\Gamma), \nu_\mu)$ is a $\mu$-boundary. Moreover, if $\mu$ has finite first moment with respect to the word metric on $G$ (induced by a finite generating set), then the measure space $(\Ends(\Gamma), \nu_\mu)$ is isomorphic to the Poisson--Furstenberg boundary of $(G, \mu)$.
\end{abstract}

\maketitle

\tableofcontents

\section{Introduction}
\label{sec:intro}

In this paper we study random walks on $\Z^n$-free groups and the boundaries of such groups. This family of groups is a particularly nice and well-studied subclass of groups acting freely on $\Lambda$-trees ({\em $\Lambda$-free groups}), where $\Lambda$ is an arbitrary ordered abelian group.

The theory of group actions on $\Lambda$-trees goes back to the work \cite{Lyndon:1963} of Lyndon who introduced abstract length functions on groups, axiomatizing Nielsen cancellation method; he initiated the study of groups with real valued length functions. Later, in \cite{Chiswell:1976}, Chiswell related such length functions with group actions on $\Z$- and $\R$-trees, providing a construction of the tree on which the group acts. At about the same time, Tits in \cite{Tits:1977} gave the first formal definition of $\R$-tree, which is a geodesic metric space with a tree-like structure. Eventually, in \cite{Morgan_Shalen:1984}, Morgan and Shalen introduced $\Lambda$-trees for an arbitrary ordered abelian group $\Lambda$ and the general form of Chiswell's construction. A $\Lambda$-tree is a metric space whose metric takes values in $\Lambda$ and is subject to certain tree axioms. The theory of group actions on such objects was consistently developed by Alperin and Bass (see \cite{Alperin_Bass:1987}), where authors state the fundamental problem: find the group theoretic information carried by a $\Lambda$-tree action, in particular, the structure of $\Lambda$-free groups. Whereas the case of Archimedean free actions, that is, when $\Lambda = \R$, is basically closed by the Rips' Theorem that describes finitely generated $\R$-free groups (see \cite{GLP:1994, Bestvina_Feighn:1995}), the general non-Archimedean case is still open, though a lot of progress was made (see \cite{Bass:1991, Guirardel:2004, KMS:2011, KMS_Survey:2013}).

\smallskip

The case when $\Lambda = \Z^n$ with the right lexicographic order received a lot of attention due to the natural combinatorial structure of $\Z^n$-trees. This structure was exploited in \cite{KMRS:2012} and \cite{KMS:2016} to obtain a description of finitely generated $\Z^n$-free groups in terms of free products with amalgamation and HNN-extensions of a particular type (see also \cite{Bass:1991}). The class of $\Z^n$-free groups is a natural generalization of free groups which contains limit groups, $\R$-free groups, etc. and which is closed under taking subgroups, free products, and amalgamated free products along maximal cyclic subgroups ($n$ is not preserved in general). All these groups are hyperbolic relative to non-cyclic maximal abelian subgroups (see \cite{Dahmani:2003, Guirardel:2004}) 
(hyperbolic if all maximal abelian subgroups are cyclic), coherent, with nice algorithmic properties.

\smallskip

Let $(X, d)$ be a $\Lambda$-tree. One can define an equivalence relation on maximal linear subtrees of $X$, which are called {\em $X$-rays}. The equivalence classes under this relation are called {\em ends} of $X$ and the subset of {\em open ends}, that is, ends not ending at a point of $X$, is denoted by $\Ends(X)$. An {\em open cut} of $X$ is a pair $(X_0, X_1)$ of open subtrees $X_0$ and $X_1$ of $X$ such that $X = X_0 \bigsqcup X_1$. An open cut $(X_0, X_1)$ corresponds to a pair of open ends $(e_0, e_1)$ respectively of $X_0$ and $X_1$ and we denote the set of all such pairs $(e_0, e_1)$ arising from open cuts by $\Cuts(X)$ (see Subsection \ref{subs:lambda_trees} for details).

Now that the sets $\Ends(X)$ and $\Cuts(X)$ are defined, we can introduce the set
$$\overbar X \colon = X \cup \Ends(X) \cup \Cuts(X).$$
In the case when $\Lambda = \Z^n$ the set $\overbar X$ can be equipped with a metric, so that $\overbar X$ is compact in the topology induced by this metric (see (\ref{eq:d_f}), Subsection \ref{subs:compact_general}).

The main result of the paper is formulated below.

\begin{theorem}
\label{th:Poisson}
Let $G$ be a countable non-abelian group acting freely and without inversions on a $\Z^n$-tree $\Gamma$, for some $n \in \N$. Assume also that the action of $G$ is minimal, that is, $\Gamma$ has no proper $G$-invariant subtrees. Then, for every non-degenerate probability measure $\mu$ on~$G$: 
\begin{enumerate}
\item there is a unique $\mu$-stationary measure $\nu_\mu$ on the space $\Ends(\Gamma)$ of ends of~$\Gamma$; this measure is continuous (i.e. it takes zero value on each point); the measure space $(\Ends(\Gamma), \nu_\mu)$ is a $\mu$-boundary of $(G, \mu)$,

\item for almost every path $\tau = \{\tau_i\}$ of the $\mu$-random walk, all the sequences $\{\tau_i \cdot v\}$, $v \in \Gamma$, converge to a random end $\omega(\tau) \in \Ends(\Gamma)$; the distribution of the limits $\omega(\tau)$ is given by the $\mu$-stationary measure~$\nu_\mu$,

\item if $G$ is finitely generated and $\mu$ has finite first moment with respect to the word metric on $G$ (induced by a finite generating set), then the measure space $(\Ends(\Gamma), \nu_\mu)$ is isomorphic to the Poisson--Furstenberg boundary of $(G, \mu)$.
\end{enumerate}
\end{theorem}

In the proof of the above theorem we use methods of \cite{Furstenberg:1973} and \cite{CartwrightSoardi:1989}, and the structure of the proof is as follows. At first, in Section \ref{sec:compact}, we introduce a metric $d_f$ on $\overbar \Gamma$ which makes $\overbar\Gamma$ a compact metric space equipped with a continuous action of $G$ (Theorem \ref{th:the-action}). Next, we prove that for every non-degenerate measure $\mu$ on $G$ there exists a unique $\mu$-stationary measure $\nu_\mu$ on $\overbar\Gamma$ and $(\overbar\Gamma, \nu_\mu)$ is a $\mu$-boundary of $(G, \mu)$ (see Proposition \ref{prop:compact_proximal}), which implies that the space $\overbar\Gamma$ is $\mu$-proximal. Then, since the measure $\nu_\mu$ on $\overbar\Gamma$ is concentrated on $\Ends(\Gamma)$ (see Proposition \ref{prop:where-stationary-measures-lives}), we deduce that $(\Ends(\Gamma), \nu_\mu)$ is a $\mu$-boundary too, and hence $\mu$-proximal (see Corollary \ref{prop:ends_proximal}). Finally, we show that if $\mu$ has finite first moment with respect to the word metric on $G$ induced by a finite generating set, then $(\Ends(\Gamma), \nu_\mu)$ is a maximal $\mu$-boundary (Proposition \ref{prop:maximality}).

\smallskip

Observe that there are other constructions of the Poisson boundaries of $\Z^n$-free groups. First of all, since every $\Z^n$-free group $G$ is hyperbolic relative to its non-cyclic abelian subgroups (follows from results of \cite{Guirardel:2004} and \cite{Dahmani:2003}), one can study its Floyd boundary (see \cite{Gerasimov_Potyagailo:2013, Gerasimov:2012}) which is non-trivial. Hence, from \cite{Karlsson:2003} it follows that the Floyd boundary of $G$ is its Poisson boundary. Another approach is to use the fact that $G$ is $CAT(0)$ with isolated flats (see 
\cite{Bumagin_Kharlampovich:2013}), so one can construct its Poisson boundary using results of  \cite{Karlsson_Margulis:1999} (see also \cite{Nevo_Sageev:2013}).

The advantage of the construction presented here is that it comes naturally from the action of a $\Z^n$-free group $G$ on the underlying $\Z^n$-tree without reference to general results about relatively hyperbolic groups and $CAT(0)$-groups, and provides a description of the boundary in terms of the underlying $\Z^n$-tree.

\bigskip

{\bf Acknowledgments:} The authors are very grateful to Vadim Kaimanovich and to Anders Karlsson for many helpful comments and suggestions. The authors also acknowledge support from the Swiss National Science Foundation. The first author is supported by RNF grant 14-11-00581.

\section{Preliminaries: Poisson--Furstenberg boundaries}
\label{subs:poisson_bound}

\subsection{Random walks and Poisson--Furstenberg boundaries}

Let $G$ be a countable group and let $\mu$ be a probability measure on~$G$. The measure $\mu$ is {\em non-degenerate} if its support generates $G$ as a semigroup. The {\em right-hand random walk on $G$ with distribution~$\mu$} (or, briefly, {\em $\mu$-random walk}) is the time-homogeneous Markov chain whose state space is $G$, the transition probabilities are given by $P(g, h) = \mu(g^{-1}h)$, and the initial distribution is concentrated at the identity~$1_G$ of the group. Realizations of this process are called {\em paths} of the random walk. 

The {\em path space} of the random walk $(G, \mu)$ is the measure space $(G^{\N_0}, P_\mu)$ ($\N_0 \colon = \{0, 1, \ldots\}$), where the measure $P_\mu$ is determined by the following condition. Let $C_{g_0, \ldots, g_p}$, $p \in \N_0$, denote the set of paths
\[
C_{g_0,\ldots,g_p} \colon = \left\{\{\tau_i\}_{i \in \N_0} \in {G^{\N_0}} \mid (\tau_0, \ldots, \tau_p) = (g_0, \ldots, g_p)\right\}.
\]
Then 
\[
P_\mu(C_{g_0, \ldots, g_p}) = \delta_{1_G}(g_0) \times \mu(g_0^{-1} g_1) \times \cdots \times \mu(g_{p-1}^{-1} g_p),
\]
where $\delta_{1_G}$ is the probability measure on~$G$ with $\delta_{1_G}(\{1_G\}) = 1$. The $\sigma$-algebra of $P_\mu$-measurable subsets is the $P_\mu$-completion of the $\sigma$-algebra generated by the {\em cylindrical sets\/} $C_{g_0, \ldots, g_p}$. Since $G$ is countable, the path space $(G^{\N_0}, P_\mu)$ is a Lebesgue space.

There are several ways to define the Poisson--Furstenberg boundary for random walks. In particular, the Poisson--Furstenberg boundary of the random walk $(G, \mu)$ can be defined as the {\it quotient measure space\/} of the path space $(G^{\N_0}, P_\mu)$ with respect to the {\it measurable hull\/} ($=$ {\em measurable envelope\/}) of the partition~$\zeta$ whose elements are the classes of the equivalence relation~$\thicksim$ (on the set $G^{\N_0}$ of paths) defined as follows:
\begin{equation}
\label{def:thicksim}
(\tau_i)_{i \in \N_0} \thicksim(\tau'_i)_{i \in \N_0} ~\Longleftrightarrow~\exists k, k' \colon \, \tau_{k+j} = \tau'_{k'+j} ~\forall j > 0.
\end{equation}
The measurable hull and hence the Poisson--Furstenberg boundary are objects defined {\em modulo subsets of measure~$0$} (${\rm mod}\ 0$). 

The partition $\zeta$ is invariant under the action of $G$ on $G^{\N_0}$ defined by the rule $g(\tau_0, \tau_1, \ldots) = (g \tau_0, g \tau_1, \ldots)$. This action induces a canonical action of~$G$ on the Poisson--Furstenberg boundary. 

\subsection{Furstenberg $\mu$-boundaries}

Our exposition of the theory of Furstenberg $\mu$-boundaries is based on \cite{Furstenberg:1973}.

An {\em action} of a group $G$ on a topological space~$M$ is a homomorphism from $G$ to the group $\operatorname{Homeo}(M)$ of all homeomorphisms of~$M$. An action is said to be ({\em topologically}) {\em minimal} if each orbit of the action is dense (or, equivalently, if the action has no proper closed invariant subsets). A~space endowed with an action of the group~$G$ is called a {\em $G$-space\/}. A map $f \colon X \to Y$ of $G$-spaces is said to be {\em equivariant\/} if $g(f(x)) = f(g(x))$ for all $x \in X$, $g \in G$.

We denote the space of all probability measures on $G$ by $\PP(G)$ and the space of all \footnote{In \cite{Furstenberg:1973}, the space of {\em compact-regular} probability measures is considered. We are interested in Polish (separable complete metric) spaces or Borel subsets thereof. All Borel probability measures on these spaces are compact-regular.} Borel probability measures on $M$ by $\PP(M)$. We endow $\PP(M)$ with the {\it weak* topology\/} (by duality with the space $C_b(M)$ of all bounded continuous functions on $M$). A~subbase for the weak* topology is the collection of sets of the form 
$$\left\{\nu\in\PP(M)\ \bigg|\ a< \left(\int_M f \,d\nu\right) < b\right\},\quad \text{where}~ f \in C_b(M),\ a, b \in \R.$$ 
Thus, a sequence $\{\nu_i\}_{i \in \N_0}$ of Borel probability measures converges with respect to the weak* topology ({\it converges weakly\/}) to a measure $\nu$ if and only if for all $f \in C_b(M)$ the numerical sequence $\left\{\int_M f \,d\nu_i\right\}_{i \in \N_0}$ converges to~$\int_M f \,d\nu$.

Any action of $G$ on $M$ induces an action of $G$ on $\PP(M)$: 
$$(g \cdot \theta)(E) = \theta (g^{-1} \cdot E).$$ 
If $x \in M$, then the {\em Dirac measure\/} $\delta_x \in \PP(M)$ is the measure defined by $\delta_x(E) = 1$ if $x\in{E}$ and $\delta_x(E) = 0$ otherwise (for a measurable~$E$). We say that a measure is {\em continuous} if it takes zero value on each point.

Suppose that $\mu \in \PP(G)$. A measure $\nu \in \PP(M)$ is said to be {\em $\mu$-stationary} if
$$\mu \ast \nu = \sum_{g \in G} (g \cdot \nu)\mu(g) = \nu.$$

\begin{asser}[\cite{Furstenberg:1973}]
\label{asser:stationary-measure-exists}
Let $G$ be a countable group acting on a compact metric space $M$ and let $\mu \in \PP(G)$ be an arbitrary measure. Then the set of $\mu$-stationary measures in $\PP(M)$ is non-empty.
\end{asser}

\begin{lemma}[\cite{KaimanovichMasur:1996}]
\label{lem:KaiMas}
Let $G$ be a countable group acting on a space $M$, $\mu$ a non-degenerate probability measure on $G$ and $\nu \in \PP(M)$ be a $\mu$-stationary measure. Let $E \subset M$ be a measurable subset such that for every $g \in G$ we have either $g \cdot E = E$, or $(g \cdot E) \cap E = \varnothing$. Suppose further that there is an infinite family of pairwise-disjoint sets of the form $g \cdot E$, where $g \in G$. Then $\nu(E) = 0$. In particular, if the orbit $G \cdot x$ of every point $x \in M$ is infinite, then $\nu$ is continuous.
\end{lemma}

\begin{lemma}
\label{lem:minimal-stationary} 
Let $G$ be a countable group acting minimally on a topological space $M$, $\mu$ be a non-degenerate measure on $G$, and $\nu \in \PP(M)$ be a $\mu$-stationary measure. Let $E \subset M$ be a non-empty open subset. Then $\nu(E) > 0$.
\end{lemma}
\begin{proof} 
Observe that by definition, the $\mu$-stationary measure $\nu$ is $\mu^{\ast k}$-stationary for every $k \in \N$. Consequently, for every $k \in \N$ we have
\begin{equation}
\label{eq:45}
\nu(E) = \sum_{g \in G} \mu^{\ast k}(g)(g \cdot \nu)(E) = \sum_{g \in G} \mu^{\ast k}(g) \nu(g^{-1} \cdot E).
\end{equation}
Suppose that $\nu(E) = 0$. Then \eqref{eq:45} implies that $\nu(g^{-1} \cdot E) = 0$ whenever there exists $k \in \N$ such that $\mu^{\ast k}(g) > 0$. Hence, $\nu(g^{-1} \cdot E) = 0$ for each $g \in G$ because $\mu$ is non-degenerate. On the other hand, since $G$ acts on $M$ minimally, and $E$ is open and non-empty, it follows that for each $x \in M$ there exists $g \in G$ such that $g \cdot x \in E$. Hence, $$\bigcup_{g \in G} g^{-1} \cdot E = \bigcup_{g \in G} g \cdot E = M.$$ Since $G$ is countable, it follows that
\begin{equation*}
1 = \nu(M) = \nu\left(\bigcup_{g \in G} g^{-1} \cdot E\right) \leqslant \sum_{g \in G} \nu(g^{-1} \cdot E) = 0.
\end{equation*}
This contradiction proves that $\nu(E) > 0$.
\end{proof}

\begin{theorem}[\cite{Furstenberg:1973,KaimanovichMasur:1996}]
\label{thm:alw-conv}
Let $G$ be a countable group acting on a compact metric space $M$. Take $\mu \in \PP(G)$ and let $\nu \in \PP(M)$ be a $\mu$-stationary measure. Then, for $P_\mu$-almost every (a.\,e.) path $\tau  = \{\tau_i\}_{i \in \N_0}$ of the $\mu$-random walk , the sequence $\{\tau_i \cdot \nu\}_{i \in \N_0}$ converges to some measure $\lambda(\tau) \in \PP(M)$. 
\end{theorem}

\begin{defn}
Suppose that $M$ is a compact metric $G$-space, $\mu \in \PP(G)$, and $\nu \in \PP(M)$ is a $\mu$-stationary measure. The pair $(M, \nu)$ is called a {\em Furstenberg $\mu$-boundary} for~$G$ if for a.\,e.~path $\tau = \{\tau_i\}_{i \in \N_0}$ of the right $\mu$-random walk, the sequence of measures $\{\tau_i \cdot \nu\}_{i \in \N_0}$ converges to some Dirac measure $\delta_{\omega(\tau)}$, where $\omega(\tau) \in M$.
\end{defn}

\begin{defn}
Let $\mu$ be a probability measure on~$G$ and let $(X, \nu)$ be the Poisson--Furstenberg boundary of the pair~$(G, \mu)$. The quotient measure space of $(X, \nu)$ with respect to a $G$-invariant measurable partition will be called a {\em $\mu$-boundary in the sense of Kaimanovich} for~$G$.
\end{defn}

\begin{asser}
\label{asser:bnd}
Let $G$ be a countable group acting on a metric space~$M$. Take a $\mu \in \PP(G)$ and let $\nu \in \PP(M)$ be a $\mu$-stationary measure. Assume that for $P_\mu$-a.\,e.~path $\{\tau_i\}_{i \in \N_0}$ of the $\mu$-random walk the sequence of measures $\{\tau_i \cdot \nu\}_{i \in \N_0}$ converges weakly to the Dirac measure $\delta_x$ for some $x \in M$. Let 
$$\bnd \colon G^{\N_0} \to M$$ 
be the {\rm(}\/$P_\mu$-a.\,e.~defined\/{\rm)} map that sends the path $\{\tau_i\}_{i \in \N_0}$ to $x$ whenever $\{\tau_i \cdot \nu\}_{i \in \N_0}$ converges to $\delta_x$. Then $\bnd$ is $P_\mu$-measurable \footnote{Since $\bnd$ is only $P_\mu$-a.\,e.~defined, it makes sense to consider measurability with respect to the $P_\mu$-completion.} and $\bnd(P_\mu) = \nu$.
\end{asser}
\begin{proof}
Let us first show that for each bounded continuous~$f \colon M \to \R$ the composition 
$$F \colon = f \circ \bnd \colon G^{\N_0} \to \R$$ 
is measurable ($P_\mu$-measurable) and
\begin{equation}
\label{eq:Gelfand}
\int_{G^{\N_0}}{F}\ d P_\mu = \int_M{f}\ d\nu.
\end{equation}
Observe that the sequence $\{F_k\}_{k \in \N_0}$ of measurable functions 
$$F_k \colon G^{\N_0} \to \R, \quad F_k(\tau) = \int_M f \ d \tau_k(\nu)$$ 
is uniformly bounded (because $f$ is bounded) and $P_\mu$-a.\,e.~converges pointwise to~$F$ (by the assumption of the assertion). Then, by the Lebesgue bounded convergence theorem $F$ is measurable \footnote{Recall that, more generally, the pointwise limit of a sequence of measurable maps from a measurable space into a metric space is measurable.} and
$$\int_{G^{\N_0}} F \ dP_{\mu} = \lim_{k \to \infty} \int_{G^{\N_0}} F_k\ dP_{\mu} = \lim_{k \to \infty} \int_{G^{\N_0}} \left(\int_M f\ d \tau_k(\nu) \right) dP_{\mu}(\tau).$$
Observe that the distribution of $\tau_k$ is given by $\mu^{*k}$ and since $\nu$ is $\mu$-stationary we have $\nu=\mu^{*k}*\nu$, whence
\begin{multline*}
\int_{G^{\N_0}} \left(\int_M f \ d \tau_k(\nu) \right) dP_{\mu}(\tau) = \sum_{g \in G} \left(\left(\int_M f \ d g(\nu) \right)\cdot \mu^{*k}(g)\right)\\ 
= \int_M f \ d (\mu^{*k}*\nu) = \int_M f \ d \nu.
\end{multline*}
We have thus proved that $F$ is measurable and 
$$\int_{G^{\N_0}}{F}\ d P_\mu = \lim_{k \to \infty}\int_{G^{\N_0}} F_k\ dP_{\mu} = \int_M{f}\ d\nu.$$

Recall that a map $m \colon \Omega \to X$ from a measure space~$\Omega$ to a metric space~$X$ is measurable (with respect to the Borel $\sigma$-algebra on~$X$) if and only if the compositions $f \circ{m} \colon \Omega \to \R$ are measurable for all bounded continuous~$f \colon X \to \R$. (This readily follows from the fact that every metric space is perfectly normal.) It then follows from the above that $\bnd$ is measurable.

Since $\bnd$ is measurable, the Borel measure $\bnd(P_\mu)$ on $M$ is well-defined and for each bounded continuous~$f \colon M \to \R$ we have
$$\int_M{f}\ d\bnd(P_\mu) = \int_{G^{\N_0}}{f \circ \bnd}\ d P_\mu,$$ 
whence it follows by~\eqref{eq:Gelfand} that
\begin{equation}
\label{eq:Gelfand-2}
\int_M{f}\ d\bnd(P_\mu) = \int_M{f}\ d\nu.
\end{equation}
Recall that, in a metric space, distinct Borel probability measures determine distinct functionals on the space of bounded continuous functions. Therefore, \eqref{eq:Gelfand-2} implies that $\bnd(P_\mu) = \nu$ as required.
\end{proof}

Assertion~\ref{asser:bnd} means that every Furstenberg $\mu$-boundary is a quotient of the measure space $(G^{\N_0}, P_\mu)$ (as a measure space, disregarding the topology) with the quotient map $\bnd$. Moreover, since $\bnd$ sends each class of the equivalence relation~$\thicksim$ (defined in \eqref{def:thicksim}) to a point, it follows that a Furstenberg $\mu$-boundary is an equivariant quotient of the Poisson--Furstenberg boundary. This proves the following corollary.

\begin{corollary}
\label{rem:quotients}
Every Furstenberg $\mu$-boundary is a $\mu$-boundary in the sense of 

Kaimanovich.
\end{corollary}

\begin{defn} (See \cite{Furstenberg:1973}.)  
Suppose that $(M, d)$ is a metric $G$-space, $\mu \in \PP(G)$. Let 
$$\mu_{(n)} = (\mu + \mu^{*2} + \cdots + \mu^{*n})/n.$$
$M$ is a said to be {\em $\mu$-proximal} if for all $x$, $y \in M$ and $\varepsilon > 0$ we have
$$\mu_{(n)}\{g | d(gx, gy) > \varepsilon\} \xrightarrow[n \to \infty]{} 0.$$ 
$M$ is called {\em mean-proximal} if it is $\mu$-proximal for all non-degenerate $\mu \in \PP(G)$. 
\end{defn}

\begin{theorem}
\label{thm-def:mean-proximal}
Let $G$ be a countable group acting on a compact metric space~$M$. Then the following conditions are equivalent\/{\rm:}
\begin{itemize}
\item[\rm(a)] $M$ is mean-proximal.

\item[\rm(b)] For any non-degenerate $\mu \in \PP(G)$ and $\mu$-stationary $\nu \in \PP(M)$, the pair $(M, \nu)$ is a Furstenberg $\mu$-boundary for~$G$.

\item[\rm(c)] For each non-degenerate $\mu \in \PP(G)$ there exists a unique $\mu$-stationary $\nu \in \PP(M)$ and the pair $(M, \nu)$ is a Furstenberg $\mu$-boundary for~$G$.
\end{itemize}
\end{theorem}
\begin{proof}
The equivalence $(a) \Leftrightarrow (b)$ follows directly from \cite[Theorem~14.1]{Furstenberg:1973}. The implication $(a) \Rightarrow (c)$ is proved in \cite[Lemma\,3.1]{Malyutin_Vershik:2008}. The implication $(c) \Rightarrow (b)$ is trivial.
\end{proof}

\subsection{The Strip Criterion of Kaimanovich} 

We use the following corollary of the `Strip Criterion', established by Kaimanovich in \cite{Kaimanovich:2000}.


\begin{corollary}[\cite{Kaimanovich:2000}]
\label{cor:strip-criterion}
Let $G$ be a finitely generated group with the induced word metric~${|\cdot|}$. Let $\mu$ be a probability measure on $G$ with finite first moment $\sum|g| \mu(g)$. Let $\check\mu$ be the {\em reflected measure} defined by $\check\mu(g) = \mu(g^{-1})$. Let $(B_+, \lambda_+)$ and $(B_-, \lambda_-)$ be $\mu$- and $\check\mu$-boundaries {\rm(}\/in the sense of Kaimanovich\/{\rm)}, respectively. If there exists a measurable $G$-equivariant map $S$ assigning to pairs of points $(b_-, b_+) \in B_- \times B_+$ non-empty ``strips'' $S(b_-, b_+) \subset G$ such that for $(\lambda_- \times \lambda_+)$-almost every $(b_-,b_+) \in B_- \times B_+$ we have
\begin{equation}
\label{eq:thin}
\frac1i \log \card \{g \in S(b_-,b_+) \mid |g| \leqslant i\}  \xrightarrow[i \to \infty]{}0,
\end{equation}
then $(B_+, \lambda_+)$ is isomorphic to the Poisson--Furstenberg boundary of the pair $(G, \mu)$.
\end{corollary}

\begin{remark}
\label{rem:empty-strips}
Note that, under the Strip Criterion (Corollary~\ref{cor:strip-criterion}), the ``strips'' $S(b_-,$ $b_+)$ are required to be
\begin{enumerate}
\item[{\rm (i)}] all non-empty,

\item[{\rm (ii')}] $(\lambda_- \times \lambda_+)$-almost surely ``thin''.
\end{enumerate}

Clearly, since the strips are allowed to meet the ``thinness'' requirement $(\lambda_- \times \lambda_+)$-almost surely (not surely), we can handle the ``non-emptiness'' property in the same way. In other words, in order to use the Strip Criterion it suffices to construct a (measurable, equivariant) map $S' \colon B_- \times B_+ \to 2^G$ with strips, which are
\begin{enumerate}
\item[{\rm (i')}] $(\lambda_- \times \lambda_+)$-almost surely non-empty,

\item[{\rm (ii')}] $(\lambda_- \times \lambda_+)$-almost surely thin.
\end{enumerate}
This is clear, because we can pass from $S'$ to a map $S$ with the property~{\rm (i)} by setting $S(b_-, b_+) = G$ if $S'(b_-, b_+) = \varnothing$ and $S(b_-, b_+) = S'(b_-, b_+)$ otherwise.

Note also that, having a map with the properties {\rm (i')} and {\rm (ii')}, we can replace all non-thin strips by empty ones and thus obtain a map that has the property~{\rm (i)} and the property
\begin{enumerate}
\item[{\rm (ii)}] all strips are ``thin''.
\end{enumerate}
\end{remark}

\section{Preliminaries: $\Z^n$-free groups}
\label{subs:Z^n-free_gps}

Following \cite{Alperin_Bass:1987} we give some basic definitions from the theory of $\Lambda$-trees and then consider the case when $\Lambda = \Z^n$. All the details can be found in \cite{Alperin_Bass:1987} and \cite{Chiswell:2001}.

\subsection{$\Lambda$-trees}
\label{subs:lambda_trees}

A set $\Lambda$ equipped with addition ``$+$'' and a partial order ``$\leqslant$'' is called a {\em partially ordered} abelian group if
\begin{enumerate}
\item[(1)] $\langle \Lambda, + \rangle$ is an abelian group,

\item[(2)] $\langle \Lambda, \leqslant \rangle$ is a partially ordered set,

\item[(3)] for all $\alpha, \beta, \gamma \in \Lambda$, $\alpha \leqslant \beta$ implies $\alpha + \gamma \leqslant \beta + \gamma$.
\end{enumerate}

An abelian group $\Lambda$ is called {\em orderable} if there exists a linear order ``$\leqslant$'' on $\Lambda$, satisfying the condition (3) above. In general, the ordering on $\Lambda$ is not unique.

For elements $\alpha, \beta \in \Lambda$, the {\em closed segment} $[\alpha, \beta]$ is defined by
$$[\alpha, \beta] = \{\gamma \in \Lambda \mid \alpha \leqslant \gamma \leqslant \beta \}.$$
A  subset $C \subset \Lambda$ is called {\em convex} if for every $\alpha, \beta \in C$ the set $C$ contains $[\alpha, \beta]$. In particular, a subgroup $C$ of $\Lambda$ is convex if $[0, \beta] \subset C$ for every positive $\beta \in C$.

Let $\Lambda$ be an ordered abelian group. A {\em $\Lambda$-metric space} is a pair $(X, d)$, where $X$ is a non-empty set and $d$ is a {\em $\Lambda$-metric} on $X$, that is, a function $d \colon X \times X \to \Lambda$ satisfying the usual metric properties. $\Lambda$-metric spaces for an arbitrary ordered abelian group $\Lambda$ were first introduced by Morgan and Shalen in \cite{Morgan_Shalen:1984}.

If $\Lambda_0$ is a convex subgroup of $\Lambda$, for any point $x_0 \in X$ the subset
$$X_0 = \{y \in X \mid d(x, y) \in \Lambda_0\}$$
of $X$ is a $\Lambda_0$-metric space with respect to the metric $d_0 = d_{\mid_{X_0}}$, called a {\em $\Lambda_0$-metric subspace} of $X$.

If $(X, d)$ and $(X', d')$ are $\Lambda$-metric spaces, an {\em isometry} from $(X, d)$ to $(X', d')$ is a mapping $f \colon X \rightarrow X'$ such that $d(x, y) = d'(f(x), f(y))$ for all $x, y \in X$.

A {\em (geodesic) segment} in a $\Lambda$-metric space is the image of the isometry $\alpha \colon [a, b]_\Lambda \rightarrow X$, where $[a,b]_\Lambda$ is a segment in $\Lambda$ between some $a, b \in \Lambda$. The endpoints of the segment are $\alpha(a), \alpha(b)$. 

We call a $\Lambda$-metric space $(X, d)$ {\em geodesic} if for all $x, y \in X$, there is a segment in $X$ with the endpoints $x, y$. We call $(X, d)$ {\em geodesically linear} if for all $x, y \in X$, there is a unique segment in $X$ whose set of endpoints is $\{x, y\}$. We denote such a segment by $[x, y]$.

A {\em $\Lambda$-tree} is a non-empty $\Lambda$-metric space $(X, d)$ such that
\begin{enumerate}
\item[(T1)] $(X, d)$ is geodesically linear,

\item[(T2)] if $x, y, z \in X$, then $[x,y] \cap [x,z] = [x,w]$ for some $w \in X$; this $w$ is unique and we write $w = Y(y,x,z)$,

\item[(T3)] if $x, y, z \in X$ and $[x,y] \cap [y,z] = \{y\}$, then $[x,y] \cup [y,z] = [x,z]$.
\end{enumerate}

A non-empty convex subset $X_0$ of a $\Lambda$-tree $X$ is called a {\em subtree}. Every subtree of a $\Lambda$-tree is obviously a $\Lambda$-tree itself (the axioms (T1)--(T3) hold for $X_0$ since it is convex). $X_0$ is {\em closed} if its intersection with any closed segment of $X$ is either empty, or a closed segment, otherwise $X_0$ is {\em open}.

A $\Lambda$-tree is called {\em linear} if it is isometric to a subtree of $\Lambda$.

Let $(X, d)$ be a $\Lambda$-tree. A point $e \in X$ is an {\em end point} of $X$ if, whenever $e \in [x,y] \subset X$, we have either $e = x$, or $e = y$. A {\em linear subtree from $x \in X$ (or, originating from $x$)} is any linear subtree $L$ of $X$ such that $x$ is an end point of $L$. Observe that $L$ carries a natural linear ordering (for $y, z \in L:\ y \leqslant z \Longleftrightarrow d(x, y) \leqslant d(x, z)$) with $x$ as the least element. If there exists a maximal point $y \in L$ with respect to this ordering, then $L$ is just a closed segment $[x, y]$ in $X$. If $v \in L$, then $L_v = \{z \in L \mid v \leqslant z\}$ is a linear subtree from $v$ and $L = [x, v] \cup L_v,\ [x, v] \cap L_v = 
\{v\}$. A maximal linear subtree from $x$ in $X$ is called an {\em $X$-ray} from $x$. By \cite[Proposition 2.22]{Alperin_Bass:1987}, the relation ``$L \cap L' = L_v = L'_v$ for some $v$'' is an equivalence relation on the set of $X$-rays. The equivalence classes under this relation are called {\em ends} of $X$. We distinguish {\em closed ends} and {\em open ends} of $X$: an end $e$ is closed if all $X$-rays associated with $e$ (ending at $e$) are closed segments $[x, e]$, otherwise $e$ is open. The set of open ends of $X$ will be denoted by $\Ends(X)$.

For any $x \in X$ and $e \in \Ends(X)$ there exists a unique $X$-ray from $x$ ending at $e$ which we denote by $[x, e \rangle$. Similarly, if $e_1, e_2 \in \Ends(X)$ are distinct ends of $X$, then there exists a unique linear subtree of $X$, which we denote by $\langle e_1, e_2 \rangle$, connecting the open ends $e_1$ and $e_2$ (see \cite[Proposition 2.24]{Alperin_Bass:1987}). 

An {\em open cut} of $X$ is a pair $(X_0, X_1)$ of open subtrees $X_0$ and $X_1$ of $X$ such that $X = X_0 \bigsqcup X_1$. By \cite[Proposition 2.26]{Alperin_Bass:1987}, every open cut $(X_0, X_1)$ uniquely corresponds to a pair of open ends $(e_0, e_1)$ respectively of $X_0$ and $X_1$. Denote the set of all such pairs $(e_0, e_1)$ arising from open cuts by $\Cuts(X)$.

Suppose $Y \subseteq X$ is a subtree of $X$. Then all open ends of $Y$ are also open ends of $X$, from which it follows that $\Cuts(Y) \subseteq \Cuts(X)$. At the same time, not all $Y$-rays are $X$-rays, so in general $\Ends(Y) \not\subseteq \Ends(X)$.

Next, if $x, y, z \in X$, then $[x, y] \cap [x, z] = [x, w]$ for some $w \in X$ by definition of $\Lambda$-tree. We denote the point $w$ by $Y(x, y, z)$. By \cite[Proposition 2.11]{Alperin_Bass:1987}, $Y(x, y, z)$ does not depend on the order in the triple $\{x, y, z\}$, that is, for example, $Y(y, x, z) = Y(x, y, z)$. It also follows that $Y(x, y, z) = [x, y] \cap [x, z] \cap [y, z]$. We are going to call $Y(x, y, z)$ the {\em median of the triple $\{x, y, z\}$}.

\smallskip

We say that a group $G$ acts on a $\Lambda$-tree $X$ if any element $g \in G$ defines an isometry $g \colon X \rightarrow X$. An action on $X$ is {\em non-trivial} if there is no point in $X$ fixed by all elements of $G$. Note that every group has a {\em trivial action} on any $\Lambda$-tree, when all group elements act as identity. An action of $G$ on $X$ is {\em minimal} if $X$ does not contain a proper $G$-invariant subtree.

Observe that an action of $G$ on $X$ induces actions respectively on $\Ends(X)$ and $\Cuts(X)$.

Next, $g \in G$ is called {\em elliptic} if it has a fixed point. An isometry $g \in G$ is called an {\em inversion} if it does not have a fixed point, but $g^2$ does. If $g$ is not elliptic and not an inversion, then it is called {\em hyperbolic}. For a hyperbolic element $g \in G$ define the characteristic set
$$Axis(g) = \{ p \in X \mid [g^{-1} \cdot p, p] \cap [p, g \cdot p] = \{p\} \},$$
which is called the {\em axis of $g$}. $Axis(g)$ meets every $\langle g \rangle$-invariant subtree of $X$.

A group $G$ acts {\em freely} and {\em without inversions} on a $\Lambda$-tree $X$ if for all $1_G \neq g \in G$, $g$ acts as a hyperbolic isometry. In this case we also say that $G$ is {\em $\Lambda$-free}.  Observe that, if $G$ is $\Lambda$-free and $f,g$ are hyperbolic, then $[f, g] = 1_G$ implies $Axis(f) = Axis(g)$, hence we denote $Axis(g)$ by $Axis(C_G(g))$.

\subsection{Groups acting on $\Z^n$-trees}
\label{subs:Z^n_trees}

Now suppose $\Lambda = \Z^n$ is considered with the right lexicographic order and let $(X, d)$ be a $\Z^n$-tree. Observe that $\Z^n$ contains the minimal positive element $(1, 0, \ldots, 0)$ which we are going to denote by $1$ (it will be clear from the context whether $1$ denotes an element of $\Z^n$, or a natural number).

Let $k \in [1, n]$. We say that $x, y \in X$ are {\em $\Z^k$-equivalent} ($x \sim_{\Z^k} y$) if $d(x, y) \in \Z^k$, that is, $d(x, y) = (a_1, \ldots, a_k, 0, \ldots, 0)$. From metric axioms it follows that ``$\sim_{\Z^k}$'' is an equivalence relation and we call the corresponding equivalence classes {\em maximal $\Z^k$-subtrees} of $X$. Denote by $\Xi_k(X)$ the set of all maximal $\Z^k$-subtrees of $X$.

Observe that $X / \sim_{\Z^k}$ is a $\Z^{n-k}$-tree with a metric $d'$ induced from $d$ (see, for example, \cite{KMS_Survey:2013} for details). The set of vertices of $X / \sim_{\Z^k}$ can be naturally identified with $\Xi_k(X)$, so one can lift $d'$ to the metric $d_k : \Xi_k(X) \to \Z^{n-k}$ on $\Xi_k(X)$ as follows: if $\pi : X \to X / \sim_{\Z^k}$ is the canonical projection, then for $S, T \in \Xi_k(X)$ define $d_k(S, T) \colon = d'(\pi(S), \pi(T))$.

\smallskip

Let $G$ be a group acting freely and without inversions on $X$ (in this case $G$ is called {\em $\Z^n$-free}). It is not hard to discover the structure of $G$ using Bass-Serre theory. By contracting elements of $\Xi_{n-1}(X)$ to points we obtain from $X$ a $\Z$-tree $Y$ equipped with the action of $G$ inherited from $X$. From Bass-Serre Theory it follows that $\Psi_G = Y / G$ is a graph in which vertices and edges correspond to $G$-orbits of vertices and edges in $Y$ and $G$ is isomorphic to the fundamental group of the graph of groups associated with $\Psi_G$. More precisely, if $v$ ranges through the set of vertices of $\Psi_G$, then $G$ can be obtained from groups $G_v$ (each $G_v$ is a 
$\Z^{n-1}$-free group isomorphic to the $G$-stabilizer of some vertex of $Y$) by means of amalgamated free products and HNN extensions along free abelian groups $G_e$ of rank not greater than $n-1$. Here, every $G_e$ is isomorphic to the $G$-stabilizer of some edge (an ordered pair of adjacent vertices) of $Y$.

In the case when $G$ is finitely generated, $\Psi_G$ is finite (see, for example, \cite{KMS:2016}). Now, since $G$ and all $G_e$ are finitely generated, and $\Psi_G$ is finite, it follows that all $G_v$ are also finitely generated (see, \cite[Proposition 29, Proposition 35]{Cohen}). Using induction on $n$ one can similarly show that the $G$-stabilizer of any maximal $\Z^k$-subtree of $X$, where $k \in [1, n-1]$, is also finitely generated. In particular, it follows that if $G$ is finitely generated, then it is finitely presented.

If $G$ is finitely generated and $H \leqslant G$ is also finitely generated, then by \cite[Theorem 65]{KMS_Survey:2013} it follows that $H$ is quasi-isometrically embedded into $G$. We are going to use this fact in the proof of Proposition \ref{prop:maximality}.

\smallskip

In our further investigations we are going to need the following technical definition (see \cite{KMS:2011} for details).

For every $v_0, v_1 \in X$ such that $d(v_0, v_1) = 1$ we call the ordered pair $(v_0, v_1)$ the {\em edge} from $v_0$ to $v_1$. Here, if $e = (v_0, v_1)$, then denote $v_0 = o(e),\ v_1 = t(e)$ which are respectively the {\em origin} and {\em terminus} of $e$. Denote by $\Ed(X)$ the set of edges of $X$. If $Y \subseteq X$ is a subtree of $X$, then we can define $\Ed(Y)$ as above by replacing $X$ with $Y$ since the metric on $Y$ is induced from $X$. Obviously, $\Ed(Y) \subseteq \Ed(X)$.

\section{From $\Z^n$-trees to metric spaces}
\label{sec:compact}

Let $\Gamma$ be a countable $\Z^n$-tree with a designated point $\base$. In this section we construct a compact metric space $\overbar{\Gamma}$ associated with $\Gamma$. Next, we study various properties of $\overbar\Gamma$ and show, in particular, that any isometric action of a group $G$ on $\Gamma$ extends to a continuous action of $G$ on $\overbar\Gamma$.

\subsection{Compactification of a $\Z^n$-tree}
\label{subs:compact_general}

Let $(\Gamma, d)$ be a countable $\Z^n$-tree. Define
$$\overbar\Gamma: = \Gamma \cup \Ends(\Gamma) \cup \Cuts(\Gamma),$$
where $\Ends(\Gamma)$ and $\Cuts(\Gamma)$ are the sets of open ends and cuts of $\Gamma$ (see Section \ref{subs:lambda_trees} for the definitions).

Recall that if $x \in \Gamma$ and $e \in \Ends(\Gamma)$, then $[x, e \rangle$ is the unique $\Gamma$-ray from $x$ to $e$. Similarly, $\langle e_1, e_2 \rangle$ is the unique linear subtree connecting distinct open ends $e_1, e_2$ of $\Gamma$.

Now, let $a, b \in \overbar\Gamma$. If $a = b$, in the case when $a \notin \Gamma$, define $\lfloor a, b \rfloor = \varnothing$, and when $a \in \Gamma$, define $\lfloor a, b \rfloor = \{a\} = \{b\}$. Otherwise, if $a \neq b$, then we define $\lfloor a, b \rfloor$ to be a linear subtree of $\Gamma$ as follows.
\begin{itemize}
\item If $a, b \in \Gamma$, then $\lfloor a, b \rfloor = [a, b]$.

\item If $a, b \in \Ends(\Gamma)$, then $\lfloor a, b \rfloor = \langle a, b \rangle$.

\item If $a \in \Gamma,\ b \in \Ends(\Gamma)$, then $\lfloor a, b \rfloor = [a, b \rangle$.

\item Let $a \in \Gamma$ and $b = (e_0, e_1) \in \Cuts(\Gamma)$, where $e_i \in \Ends(\Gamma_i),
\Gamma = \Gamma_0 \bigsqcup \Gamma_1$ and $a \in \Gamma_0$. Then, define $\lfloor a, b \rfloor = [a, e_0 \rangle$.

\item Let $a \in \Ends(\Gamma)$ and $b = (e_0, e_1) \in \Cuts(\Gamma)$, where $e_i \in \Ends(\Gamma_i),
\Gamma = \Gamma_0 \bigsqcup \Gamma_1$. Then, $a$ is an end of one of the subtrees $\Gamma_0,\ \Gamma_1$, say $a \in \Ends(\Gamma_0)$, so, we define $\lfloor a, b \rfloor = \langle a, e_0 \rangle$.

\item Let $a = (e_0, e_1)$, where $e_i \in \Ends(\Gamma_i), \Gamma = \Gamma_0 \bigsqcup \Gamma_1$, and $b = (f_0, f_1) \in \Cuts(\Gamma)$. It follows that $b$ is an open cut of one of the subtrees $\Gamma_0,\ \Gamma_1$, say $b \in \Cuts(\Gamma_0)$. Now, since $e_0 \in \Ends(\Gamma_0)$ and $b \in \Cuts(\Gamma_0)$, replacing $\Gamma$ with $\Gamma_0$ in the definition case above we can assume that $\lfloor e_0, b \rfloor$ is defined for $\Gamma_0$ and set $\lfloor a, b \rfloor = \lfloor e_0, b \rfloor$.
\end{itemize}

It is easy to see from the definition above that $\lfloor a, b \rfloor = \lfloor b, a \rfloor$ for any $a, b \in \overbar\Gamma$. Also, if $\lfloor a, b \rfloor \neq \varnothing$, then $a$ and $b$ can be considered as ends of $\lfloor a, b \rfloor$ (closed if points in $\Gamma$, open otherwise).

\begin{lemma}
\label{le:all-ends}
If $a \in \Gamma$, then the map $\psi \colon b \mapsto \lfloor a, b \rfloor$ establishes a one-to-one correspondence between $\overbar\Gamma$ and the set of all linear subtrees from $a$ in $\Gamma$.
\end{lemma}
\begin{proof} 
From the definition above it follows that $\lfloor a, b_1 \rfloor \neq \lfloor a, b_2 \rfloor$ if $b_1 \neq b_2$. Hence, $\psi$ is injective. In order to show that it is also surjective, for each linear subtree $L$ of $\Gamma$ from $a$ we show that there exists $b \in \overbar\Gamma$ such that $L = \lfloor a, b \rfloor$.

\smallskip

If $L$ is closed, then $L = [a, b]$ for some $b \in \Gamma$. That is, $L = \lfloor a, b \rfloor$.

Now, let $L$ be open. If $L$ is maximal, then it is a $\Gamma$-ray and $L = [a, b \rangle$ for some $b \in \Ends(\Gamma)$ and we have $L = \lfloor a, b \rfloor$. 

If $L$ is not maximal, then there exists a subtree $\Gamma'$ of $\Gamma$ such that $L$ is a $\Gamma'$-ray from $a$ (for example, $L$ itself is such a subtree). The union of all such subtrees $\Gamma'$ is a subtree $\Gamma_0$ of $\Gamma$ since every $\Gamma'$ contains $L$. In fact, from the construction of $\Gamma_0$ it follows that it is the maximal subtree (by inclusion) of $\Gamma$ such that $L$ is a $\Gamma_0$-ray from $a$.

Observe that $\Gamma_0$ is open. Indeed, since $L$ is not maximal in $\Gamma$, it can be extended to a $\Gamma$-ray $L'$ from $a$ and there exists $x \in L' \ssm L$. Hence, $x \notin \Gamma_0$ (because $L$ is maximal in $\Gamma_0$) and $\Gamma_0 \cap [a, x] = L$ which is not a closed segment. That is, $\Gamma_0$ cannot be closed, hence, it is open. 

Next, consider $\Gamma_1 = \Gamma \ssm \Gamma_0$. We are going to show that $\Gamma_1$ is a subtree of $\Gamma$. Observe that $L_1 = L' \ssm L$ is a linear tree contained in $\Gamma_1$ and $x$ defined above belongs to $L_1$. If $y \in \Gamma_1$ is such that $[x, y] \cap \Gamma_0 \neq \varnothing$, then $y \notin L_1$ and $\Gamma' = [y, a] \cup L$ is a subtree of $\Gamma$ such that $L$ is a $\Gamma'$-ray from $a$. But then we get a contradiction with the construction of $\Gamma_0$ since $\Gamma' \not\subseteq \Gamma_0$. The contradiction comes from the assumption that $[x, y] \cap \Gamma_0 \neq \varnothing$ and it follows that $[x, y] \in \Gamma_1$ for every $y \in \Gamma_1$. In other words, $\Gamma_1$ is convex, hence, it is a subtree of $\Gamma$. 

Moreover, $\Gamma_1$ is open. Assume on the contrary that $\Gamma_1$ is closed and consider $L_1 \subseteq \Gamma_1$. Since $\Gamma_1$ is closed, $L_1$ is also closed and there exists $z \in L_1$ such that $[a, x] \cap L_1 = [z, x]$. Now, since the metric on $\Gamma$ takes values in $\Z^n$, there exists a point $z_0 \in L'$ such that $d(a, z) = d(a, z_0) + 1$, where $1$ is the minimal positive element of $\Z^n$. Obviously, $z_0 \notin L_1$ because otherwise $[a, x] \cap L_1 = [z_0, x]$ -- a contradiction. Hence, $z_0 \in L$ and $L = [a, z_0]$ -- a contradiction with the fact that $L$ is open. 

Thus, $\Gamma_1$ is open and it follows that $b = (\Gamma_0, \Gamma_1)$ is an open cut such that $L = \lfloor a, b \rfloor$.

From the argument above it follows that for any linear subtree $L$ from $a$ there exists $b \in \overbar\Gamma$ such that $L = \lfloor a, b \rfloor$ and the statement of the lemma follows.
\end{proof}

Below we are going to use the following convention: for an edge $e \in Ed(\Gamma)$ and $a, b \in \overbar\Gamma$  we write $e \in \lfloor a, b \rfloor$ meaning that $e \in Ed \lfloor a, b \rfloor$. Also, if $a = b$, then obviously $\Ed \lfloor a, b \rfloor = \varnothing$ since $\lfloor a, b \rfloor$ in this case is either empty, or consist of a single point of $\Gamma$.

\begin{lemma}
\label{le:Y0}
If $a, b, c \in \overbar\Gamma$, then the following hold.
\begin{enumerate}
\item[(a)] If $\lfloor a, b \rfloor \neq \varnothing$ and $x \in \lfloor a, b \rfloor$, then $\lfloor a, b \rfloor = \lfloor a, x \rfloor \cup \lfloor x, b \rfloor$ and $\lfloor a, x \rfloor \cap \lfloor x, b \rfloor = \{ x \}$.

\item[(b)] If $x \in \Gamma$ and $\lfloor a, x \rfloor \cap \lfloor x, b \rfloor = \{ x \}$, then $\lfloor a, x \rfloor \cup \lfloor x, b \rfloor = \lfloor a, b \rfloor$.

\item[(c)] If $a, b, c$ are pairwise distinct, then $\lfloor a, c \rfloor \cap \lfloor c, b \rfloor \neq \varnothing$ unless $c = (e_0, e_1) \in \Cuts(\Gamma)$, where $e_i \in \Ends(\Gamma_i), \Gamma = \Gamma_0 \bigsqcup \Gamma_1$, is such that $\lfloor a, c \rfloor \subseteq \Gamma_0$ and $\lfloor c, b \rfloor \subseteq \Gamma_1$.

\item[(d)] $\lfloor a, c \rfloor \cup \lfloor c, b \rfloor$ is a subtree of $\Gamma$.

\item[(e)] $\lfloor a, b \rfloor \subseteq \lfloor a, c \rfloor \cup \lfloor c, b \rfloor$.

\item[(f)] $\Ed \lfloor a, b \rfloor \subseteq \Ed \lfloor a, c \rfloor \cup \Ed \lfloor c, b \rfloor$.

\item[(g)] $\Ed\lfloor a, b \rfloor \cap \Ed\lfloor a, c \rfloor \cap \Ed\lfloor b, c \rfloor = \varnothing$.

\item[(h)] If $c \in \Gamma$ and $\lfloor a, b \rfloor \neq \varnothing$, then 
$$\lfloor c, a \rfloor \cap \lfloor c, b \rfloor \subseteq \lfloor c, z \rfloor$$
for any $z\in \lfloor a, b \rfloor$.
\end{enumerate}
\end{lemma}
\begin{proof}
{\bf (a)} If $a, b \in \Gamma$, then by definition of ``$\lfloor\,  ,\, \rfloor$", we have $\lfloor a, b \rfloor = [a, b]$. Since $x \in [a, b]$, we have $[a, x] \cap [x, b] = \{x\}$ and $[a, b]  = [a, x] \cup [x, b]$ by definition of $\Lambda$-tree. At the same time, $\lfloor a, x \rfloor = [a, x],\ \lfloor x, b \rfloor = [x, b]$, and the statement follows.

If $a \in \Gamma$ and $b \in \Ends(\Gamma)$, then by definition of ``$\lfloor\,  ,\, \rfloor$", we have $\lfloor a, b \rfloor = [a, b \rangle$ which is a linear subtree from $a$. Since $x \in [a, b \rangle$, it follows that $[x, b \rangle$ is a linear subtree from $x$ such that we have $[a, b \rangle = [a, x] \cup [x, b \rangle$ and $[a, x] \cap [x, b \rangle = \{ x \}$ (see \cite[Definitions 2.21]{Alperin_Bass:1987}). Now the statement follows since $[a, x] = \lfloor a, x \rfloor$ and $[x, b \rangle = \lfloor x, b \rfloor$.

If $a \in \Gamma$ and $b = (e_0, e_1) \in \Cuts(\Gamma)$, where $e_i \in \Ends(\Gamma_i), \Gamma = \Gamma_0 \bigsqcup \Gamma_1$ and $a \in \Gamma_0$, then by definition of ``$\lfloor\,  ,\, \rfloor$", we have $\lfloor a, b \rfloor = [a, e_0 \rangle$. This is a linear subtree from $a$ in $\Gamma_0$, so the required statement follows from the argument given in the previous paragraph with $\Gamma$ replaced by $\Gamma_0$.

Suppose $a, b \in \Ends(\Gamma)$. Hence, by definition of ``$\lfloor\,  ,\, \rfloor$", we have $\lfloor a, b \rfloor = \langle a, b \rangle$. Since $x \in \langle a, b \rangle$, it follows that $\langle a, x] \cap [x, b \rangle = \{x\}$. Now from \cite[Proposition 2.24]{Alperin_Bass:1987}, it follows that $\langle a, x] \cup [x, b \rangle = \langle a, b \rangle$. Hence, the statement holds since $\langle a, x]  = \lfloor a, x \rfloor$ and $[x, b \rangle = \lfloor x, b \rfloor$ by definition of ``$\lfloor\,  ,\, \rfloor$". 

Suppose $a \in \Ends(\Gamma)$ and $b = (e_0, e_1) \in \Cuts(\Gamma)$, where $e_i \in \Ends(\Gamma_i),
\Gamma = \Gamma_0 \bigsqcup \Gamma_1$. Then, assuming that $a \in \Ends(\Gamma_0)$, by definition of ``$\lfloor\,  ,\, \rfloor$", we have $\lfloor a, b \rfloor = \langle a, e_0 \rangle$. Since $e_0, a \in \Ends(\Gamma_0)$, the statement follows from the argument above with $\Gamma$ replaced by $\Gamma_0$.

Finally, suppose that $a = (e_0, e_1)$, where $e_i \in \Ends(\Gamma_i), \Gamma = \Gamma_0 \bigsqcup\Gamma_1$, and $b = (f_0, f_1) \in \Cuts(\Gamma)$. Assuming that $b \in \Cuts(\Gamma_0)$, by definition of ``$\lfloor\,  ,\, \rfloor$", we have $\lfloor a, b \rfloor = \lfloor e_0, b \rfloor$ and the statement follows from the argument above with $\Gamma$ replaced by $\Gamma_0$.

\smallskip

{\bf (b)} If $a, b \in \Gamma$, then the statement holds by definition of $\Lambda$-tree.

If $a \in \Gamma$ and $b \in \Ends(\Gamma)$, then $\lfloor a, x \rfloor = [a, x],\ \lfloor x, b \rfloor = [x, b \rangle$.
In particular, $[x, b \rangle$ is a linear subtree from $x$ and $[a, x] \cap [x, b \rangle = \{ x \}$. Hence, $[a, x] \cup [x, b \rangle$ is a linear subtree from $a$ which end at $b$ (see \cite[Definitions 2.21]{Alperin_Bass:1987}). That is, $[a, x] \cup [x, b \rangle = [a, b \rangle = \lfloor a, b \rfloor$ and the statement follows.

If $a \in \Gamma$ and $b = (e_0, e_1) \in \Cuts(\Gamma)$, where $e_i \in \Ends(\Gamma_i), \Gamma = \Gamma_0 \bigsqcup \Gamma_1$ and $a \in \Gamma_0$, then by definition we have $\lfloor a, b \rfloor = [a, e_0 \rangle$. It follows that $x \in \Gamma_0$ because otherwise we have a contradiction with $[a, x] \cap [x, b \rangle = \{ x \}$. Now, the statement follows from the argument above with $\Gamma$ replaced by $\Gamma_0$.

Suppose $a, b \in \Ends(\Gamma)$. By \cite[Proposition 2.24]{Alperin_Bass:1987}, for any $z \in \Gamma$ we have $[z, a \rangle \cup [z, b \rangle = \langle a, b \rangle$ whenever $[z, a \rangle \cap [z, b \rangle = \{ z \}$. This implies the statement because we have $\lfloor x, a \rfloor = [x, a \rangle$, $\lfloor x, b \rfloor = [x, b \rangle$, and $\lfloor a, b \rfloor = \langle a, b \rangle$ by definition of `$\lfloor\,  ,\, \rfloor$'.

Suppose $a \in \Ends(\Gamma)$ and $b = (e_0, e_1) \in \Cuts(\Gamma)$, where $e_i \in \Ends(\Gamma_i),\Gamma = \Gamma_0 \bigsqcup \Gamma_1$. Assume that $a \in \Ends(\Gamma_0)$. Hence, $x \in \Gamma_0$
because otherwise we have a contradiction with $[x, a \rangle \cap [x, b \rangle = \{ x \}$. Now, the statement follows from the argument above with $\Gamma$ replaced by $\Gamma_0$.

The case when $a, b \in \Cuts(\Gamma)$ can be similarly reduced to one of the previous cases.

\smallskip

{\bf (c)} If $c \in \Gamma$, then both $\lfloor a, c \rfloor$ and $\lfloor c, b \rfloor$ are linear subtrees of $\Gamma$ from $c$. Hence, $c \in \lfloor a, c \rfloor \cap \lfloor c, b \rfloor$ and the statement follows.

Suppose $c \in \Ends(\Gamma)$ and let $x \in \lfloor a, c \rfloor,\ y \in  \lfloor c, b \rfloor$ be arbitrary (such $x$ and $y$ exist since both $\lfloor a, c \rfloor$ and $\lfloor c, b \rfloor$ are non-empty by assumption). Hence, $[x, c \rangle$ and $[y, c \rangle$ represent the same end of $\Gamma$ and by definition of ends it follows that $[x, c \rangle \cap [y, c \rangle$ is non-empty. Hence, $\lfloor a, c \rfloor \cap \lfloor c, b \rfloor$ is also non-empty.

Now, let $c = (e_0, e_1) \in \Cuts(\Gamma)$, where $e_i \in \Ends(\Gamma_i), \Gamma = \Gamma_0 \bigsqcup \Gamma_1$. Suppose that both $\lfloor a, c \rfloor$ and $\lfloor c, b \rfloor$ belong to one of the subtrees $\Gamma_0,\ \Gamma_1$, say $\Gamma_0$. Hence, repeating the argument above with $\Gamma$ replaced by $\Gamma_0$ we obtain that $\lfloor a, c \rfloor \cap \lfloor c, b \rfloor \neq \varnothing$. 

\smallskip

{\bf (d)} If at least two of the points $a, b, c$ coincide, then {\bf (d)} trivially holds. So, suppose that $a, b, c$ are pairwise distinct.

From {\bf (c)} it follows that $\lfloor a, c \rfloor \cup \lfloor c, b \rfloor$ is connected (hence, it is a subtree of $\Gamma$) unless $c = (e_0, e_1) \in \Cuts(\Gamma)$, where $e_i \in \Ends(\Gamma_i), \Gamma = \Gamma_0 \bigsqcup \Gamma_1$, is such that $\lfloor a, c \rfloor \subseteq \Gamma_0$ and $\lfloor c, b \rfloor \subseteq \Gamma_1$. Assume the latter and take arbitrary points $x \in \lfloor a, c \rfloor,\ y \in  \lfloor c, b \rfloor$. Then, by \cite[Proposition 2.26]{Alperin_Bass:1987}, we have $[x, y] = [x, e_0 \rangle \cup [y, e_1 \rangle$. Next, by definition of ``$\lfloor\,  ,\, \rfloor$", we have $\lfloor x, c \rfloor = [x, e_0 \rangle,\ \lfloor c, y \rfloor = \langle e_1, y]$, hence,
$$[x, y] = \lfloor x, c \rfloor \cup \lfloor c, y \rfloor.$$
Applying {\bf (a)} we have
$$\lfloor a, c \rfloor \cup \lfloor c, b \rfloor = (\lfloor a, x \rfloor \cup  \lfloor x, c \rfloor) \cup (\lfloor c, y \rfloor \cup \lfloor y, b \rfloor = \lfloor a, x \rfloor \cup [x, y] \cup \lfloor y, b \rfloor$$
which is connected. Hence, $\lfloor a, c \rfloor \cup \lfloor c, b \rfloor$ is a subtree of $\Gamma$.

\smallskip

{\bf (e)} Again, if at least two of the points $a, b, c$ coincide, then {\bf (e)} trivially holds. So, suppose that $a, b, c$ are pairwise distinct.

Suppose at first that one of the subtrees $\lfloor a, c \rfloor,\ \lfloor c, b \rfloor$ is contained in the other one, say $\lfloor a, c \rfloor \subset \lfloor c, b \rfloor$. If $a \in \Gamma$, then $a \in \lfloor c, b \rfloor$. Hence, from {\bf (a)} we have
$$\lfloor c, b \rfloor = \lfloor c, a \rfloor \cup \lfloor a, b \rfloor$$
and the statement follows. Suppose $a \notin \Gamma$. Then it follows that $a \in \Cuts(\Gamma)$. Indeed, $a$ cannot be an end of $\Gamma$ since for every $z \in \lfloor c, a \rfloor$, the linear subtree $\lfloor z, a \rfloor$ is not a maximal linear subtree from $z$. Let $a = (e_0, e_1)$, where $e_i \in \Ends(\Gamma_i), \Gamma = \Gamma_0 \bigsqcup \Gamma_1$. Assume, without loss of generality, that $c$ belongs to $\Gamma_0$ and $b$ belongs to $\Gamma_1$ (as points, ends, or open cuts). From {\bf (c)} it follows that there exists $x \in \lfloor c, b \rfloor \cap \lfloor a, b \rfloor$. Indeed, since $a \in \Cuts \lfloor c, b \rfloor$, the ends $e_0$ and $e_1$ uniquely correspond to each other and they cannot be separated by $b$ because $b \neq a$. By {\bf (a)} we have
$$\lfloor c, b \rfloor = \lfloor c, x \rfloor \cup \lfloor x, b \rfloor,\ \ \lfloor a, b \rfloor = \lfloor a, x \rfloor \cup \lfloor x, b \rfloor$$
and it is enough to show that $\lfloor a, x \rfloor \subset \lfloor c, x \rfloor$. But $\lfloor a, x \rfloor = \langle a, x]$ by definition of ``$\lfloor\,  ,\, \rfloor$" and for every $y \in \lfloor c, x \rfloor \cap \Gamma_0$, by \cite[Proposition 2.26]{Alperin_Bass:1987}, we have $[y, x] = [y, e_0 \rangle \cup \langle e_1, x]$. In other words,
$$\lfloor a, x \rfloor = \langle e_1, x] \subset [y, x] \subseteq \lfloor c, x \rfloor.$$

Now suppose that $\lfloor a, c \rfloor \not\subset \lfloor c, b \rfloor,\ \lfloor c, b \rfloor \not\subset \lfloor a, c \rfloor$. From {\bf (c)} we have the following two cases.

{\bf Case I.} $\lfloor a, c \rfloor \cap \lfloor c, b \rfloor \neq \varnothing$.

Let $x \in \lfloor a, c \rfloor \ssm \lfloor c, b \rfloor$ and $y \in \lfloor c, b \rfloor \ssm \lfloor a, c \rfloor$. Observe that $\lfloor x, c \rfloor$ and $\lfloor y, c \rfloor$ are linear subtrees of $\lfloor a, c \rfloor \cup \lfloor c, b \rfloor$ (which is a subtree of $\Gamma$ by {\bf (d)}). Hence, by \cite[Proposition 2.22]{Alperin_Bass:1987}, there exists $v \in \lfloor a, c \rfloor \cup \lfloor c, b \rfloor$ such that either $\lfloor a, c \rfloor \cap \lfloor c, b \rfloor = [c, v]$ in the case when $c \in \Gamma$, or $\lfloor a, c \rfloor \cap \lfloor c, b \rfloor = \langle c, v]$ in the case when $c \notin \Gamma$. In both cases, by definition of ``$\lfloor\,  ,\, \rfloor$", we have
$$\lfloor a, c \rfloor \cap \lfloor c, b \rfloor = \lfloor c, v \rfloor$$
and $\lfloor a, v \rfloor \cap \lfloor b, v \rfloor = \{ v \}$. Now, by {\bf (b)} we have $\lfloor a, v \rfloor \cup \lfloor b, v \rfloor = \lfloor a, b \rfloor$ and
$$\lfloor a, b \rfloor = \lfloor a, v \rfloor \cup \lfloor b, v \rfloor \subset \lfloor a, c \rfloor  \cup \lfloor c, b \rfloor,$$
where the last inclusion follows from {\bf (a)}.

{\bf Case II.} $c = (e_0, e_1) \in \Cuts(\Gamma)$, where $e_i \in \Ends(\Gamma_i), \Gamma = \Gamma_0 
\bigsqcup \Gamma_1$, is such that $\lfloor a, c \rfloor \subseteq \Gamma_0$ and $\lfloor c, b \rfloor \subseteq \Gamma_1$.

By {\bf (c)} there exist $x \in \lfloor a, c \rfloor \cap \lfloor a, b \rfloor$ and $y \in \lfloor c, b \rfloor \cap \lfloor a, b \rfloor$. Hence, by {\bf (a)} we have
$$\lfloor a, b \rfloor = \lfloor a, x \rfloor \cup \lfloor x, y \rfloor \cup \lfloor y, b \rfloor$$
$$\lfloor a, c \rfloor = \lfloor a, x \rfloor \cup \lfloor x, c \rfloor,\ \ \lfloor c, b \rfloor = \lfloor c, y \rfloor \cup \lfloor y, b \rfloor.$$
Hence, in order to prove $\lfloor a, b \rfloor \subseteq \lfloor a, c \rfloor \cup \lfloor c, b \rfloor$ it is enough to show that $\lfloor x, y \rfloor \subseteq \lfloor x, c \rfloor \cup \lfloor c, y \rfloor$. But we have $\lfloor x, y \rfloor = [x, y],\ \lfloor x, c \rfloor = [x, c \rangle,\ \lfloor c, y \rfloor = \langle c, y]$, so the required follows from \cite[Proposition 2.26]{Alperin_Bass:1987}.

\smallskip

{\bf (f)} Suppose on the contrary that there exists an edge $(v_0, v_1) $ in the set $\Ed \lfloor a, b \rfloor \ssm (\Ed \lfloor a, c \rfloor \cup \Ed \lfloor c, b \rfloor) $. This implies that one of the points $v_0$ and $v_1$ lies in $\lfloor c, a \rfloor \ssm \lfloor c, b \rfloor $, while the other one is in $\lfloor c, b \rfloor \ssm \lfloor c, a \rfloor $. Then $v_0$ is not in $\lfloor v_1, c \rfloor$, while $v_1$ is not in $\lfloor v_0, c \rfloor$. 

If $\lfloor v_1, c \rfloor \cap \lfloor v_0, c \rfloor $ is empty, then {\bf (c)} implies that $v_0$ and $v_1$ are in distinct $\Z$-subtrees of $\Gamma$ so that $(v_0, v_1) $ is not an edge~--- a contradiction.

If $\lfloor v_1, c \rfloor \cap \lfloor v_0, c \rfloor$ is not empty, then \cite[Proposition 2.28]{Alperin_Bass:1987} implies that
$$\lfloor v_0, v_1 \rfloor \cap \lfloor v_0, c \rfloor \cap \lfloor v_1, c \rfloor \neq \varnothing.$$
Suppose $v \in \lfloor v_0, v_1 \rfloor \cap \lfloor v_0, c \rfloor \cap \lfloor v_1, c \rfloor \neq \varnothing$. Then $v \notin \{v_0,v_1\}$ because $v_0$ is not in $\lfloor v_1, c \rfloor$, while $v_1$ is not in $\lfloor v_0, c \rfloor$. Therefore, $[v_0, v_1] = [v_0, v, v_1] $, which means that $(v_0, v_1) $ is not an edge~--- a contradiction.

\smallskip

{\bf (g)} If at least one of the intersections $\lfloor a, b \rfloor \cap \lfloor a, c \rfloor,\ \lfloor a, b \rfloor \cap \lfloor b, c \rfloor,\ \lfloor a, c \rfloor \cap \lfloor b, c \rfloor$ is empty, then $\lfloor a, b \rfloor \cap \lfloor a, c \rfloor \cap \lfloor b, c \rfloor$ is also empty and the statement follows.

Suppose that all three intersections $\lfloor a, b \rfloor \cap \lfloor a, c \rfloor,\ \lfloor a, b \rfloor \cap \lfloor b, c \rfloor,\ \lfloor a, c \rfloor \cap \lfloor b, c \rfloor$ are non-empty. By \cite[Proposition 2.28]{Alperin_Bass:1987} it follows that 
$$\lfloor a, b \rfloor \cap \lfloor a, c \rfloor \cap \lfloor b, c \rfloor \neq \varnothing.$$
Suppose $v \in \lfloor a, b \rfloor \cap \lfloor a, c \rfloor \cap \lfloor b, c \rfloor$. Hence, from {\bf (a)} we have
$$\lfloor a, b \rfloor = \lfloor a, v \rfloor \cup \lfloor v, b \rfloor,\ \ \lfloor a, c \rfloor = \lfloor a, v \rfloor \cup \lfloor v, c \rfloor,\ \ \lfloor b, c \rfloor = \lfloor b, v \rfloor \cup \lfloor v, c \rfloor$$
and
$$\lfloor a, b \rfloor \cap \lfloor a, c \rfloor = (\lfloor a, v \rfloor \cup \lfloor v, b \rfloor) \cap (\lfloor a, v \rfloor \cup \lfloor v, c \rfloor)$$
$$= \lfloor a, v \rfloor \cup (\lfloor a, v \rfloor \cap \lfloor v, c \rfloor) \cup (\lfloor b, v \rfloor \cap \lfloor v, a \rfloor) \cup  (\lfloor b, v \rfloor \cap \lfloor v, c \rfloor)$$
$$= \lfloor a, v \rfloor \cup \{ v \} \cup \{ v \} \cup \{ v \} = \lfloor a, v \rfloor.$$
Hence, apply {\bf (a)} again and obtain
$$\lfloor a, b \rfloor \cap \lfloor a, c \rfloor \cap \lfloor b, c \rfloor = \lfloor a, v \rfloor \cap \lfloor b, c \rfloor = \lfloor a, v \rfloor \cap (\lfloor b, v \rfloor \cup \lfloor v, c \rfloor)$$
$$= (\lfloor a, v \rfloor \cap \lfloor b, v \rfloor) \cup (\lfloor a, v \rfloor \cap \lfloor v, c \rfloor) = \{ v \}.$$
Since since the intersection $\lfloor a, b \rfloor \cap \lfloor a, c \rfloor \cap \lfloor b, c \rfloor$ is just a single point, the required follows.

\smallskip

{\bf (h)} Suppose on the contrary that there exists $x\in \lfloor c, a \rfloor\cap \lfloor c, b \rfloor$ such that $x\notin\lfloor c, z \rfloor$. By {\bf (e)}, we have $\lfloor c, b \rfloor \subseteq \lfloor c, z \rfloor \cup \lfloor z, b \rfloor$. Since $x\in \lfloor c, b \rfloor$ and $x\notin\lfloor c, z \rfloor$, it follows that $x\in \lfloor z, b \rfloor$. At the same time, {\bf (e)} implies that we have $\lfloor c, a \rfloor \subseteq \lfloor c, z \rfloor \cup \lfloor z, a \rfloor$. Since $x\in \lfloor c, a \rfloor$ and $x\notin\lfloor c, z \rfloor$, it follows that $x\in \lfloor z, a \rfloor$. Thus, it is shown that $x\in \lfloor z, b \rfloor \cap \lfloor z, a \rfloor$. However, $\lfloor z, b \rfloor \cap \lfloor z, a \rfloor=\{z\}$ by {\bf (a)}. Consequently, we have $x\in\{z\}$ --- a contradiction.
\end{proof}

Next, recall that in Section \ref{subs:lambda_trees} we introduced the notion of the median $Y(x, y, z)$ for any triple of points $\{x, y, z\} \subset \Gamma$. We extend this notion to $\overbar\Gamma$: for $a, b, c \in \overbar\Gamma$
\begin{itemize}
\item set $Y(a, b, c) = a$ if $a = b$ or $a = c$, and set $Y(a, b, c) = b$ if $b = c$,

\item if one of the points $a, b, c$ of $\overbar\Gamma$, say $c$, is an open cut $c = (e_0, e_1) \in \Cuts(\Gamma)$, where $e_i \in \Ends(\Gamma_i), \Gamma = \Gamma_0 \bigsqcup \Gamma_1$, such that $\lfloor a, c \rfloor \subseteq \Gamma_0$ and $\lfloor c, b \rfloor \subseteq \Gamma_1$, then set $Y(a, b, c) = c$,

\item set $Y(a, b, c) = \lfloor a, b \rfloor \cap \lfloor a, c \rfloor \cap \lfloor b, c \rfloor$ otherwise.
\end{itemize}

Observe that $Y(a, b, c)$ is correctly defined for any triple $a, b, c \in \overbar\Gamma$. Indeed, suppose $a, b$, and $c$ are all pairwise distinct and one of the intersections
$$\lfloor a, b \rfloor \cap \lfloor a, c \rfloor,\ \ \ \lfloor a, b \rfloor \cap \lfloor b, c \rfloor, \ \ \ \lfloor a, c \rfloor \cap \lfloor b, c \rfloor,$$
say $\lfloor a, c \rfloor \cap \lfloor b, c \rfloor$, is empty. In this case, by the assertion (c) of Lemma \ref{le:Y0} we have the case covered in the second part of the definition above. Now, if all three intersections $\lfloor a, b \rfloor \cap \lfloor a, c \rfloor,\ \lfloor a, b \rfloor \cap \lfloor b, c \rfloor,\ \lfloor a, c \rfloor \cap \lfloor b, c \rfloor$ are non-empty, by the proof of assertion (g) of Lemma \ref{le:Y0}, $\lfloor a, b \rfloor \cap \lfloor a, c \rfloor \cap \lfloor b, c \rfloor$ is a single point $v \in \Gamma$ and by the definition above we have $Y(a, b, c) = v$.

\smallskip

Now we define a metric on $\overbar\Gamma$. Consider the set $\Ed(\Gamma)$ which is countable by assumption. Let $f \colon \Ed(\Gamma) \to \R_+$ be a \emph{summable} positive real-valued function, that is
$$\sum_{e \in \Ed(\Gamma)} f(e) < +\infty.$$
We define the function 
\begin{equation}
\label{eq:d_f}
d_f \colon \overbar\Gamma \times \overbar\Gamma \to \R
\end{equation}
by setting
$$d_f(a,b) = \sum_{e \in \Ed \lfloor a, b \rfloor} f(e).$$
Then $d_f$ is clearly a metric on $\overbar\Gamma$ (the triangle inequality for $d_f$ follows from the assertion (f) of Lemma \ref{le:Y0}). 

\begin{lemma}
\label{lem:1-convergence}
Let $\{a_i\}$ and $\{b_i\}$ be two sequences of points in the metric space $(\overbar\Gamma, d_f)$. Then
\begin{itemize}
\item[(1)] the sequence $\{a_i\}$ is fundamental (i.e. Cauchy sequence) if and only if for every edge $e \in \Ed(\Gamma)$ there exists $N > 0$ such that $e \notin \Ed \lfloor a_j, a_k \rfloor$ whenever $j > N$ and $k > N$,

\item[(2)] the sequence $\{d_f(a_i, b_i)\}$ tends to zero if and only if for each $e \in \Ed(\Gamma)$, the set $\{i \in \N\, \mid\, e \in \Ed \lfloor a_i, b_i \rfloor\}$ is at most finite,

\item[(3)] the sequence $\{a_i\}$ converges to a point $t \in \overbar\Gamma$ if and only if for each $e \in \Ed(\Gamma)$ the set $\{i \in \N\, \mid\, e \in \Ed \lfloor a_i, t \rfloor\}$ is at most finite.
\end{itemize}
\end{lemma}
\begin{proof} 
(1) If $\{a_i\}$ is fundamental and $e \in \Ed(\Gamma)$, then there exists $N > 0$ such that $d_f(a_j, a_k) < f(e)$ whenever $j > N$ and $k > N$. It follows that $e \notin \Ed \lfloor a_j, a_k \rfloor$ whenever $j > N$ and $k > N$.

Conversely, assume that for every $e \in \Ed(\Gamma)$ there exists $N(e) > 0$ such that $e \notin \Ed \lfloor a_j, a_k \rfloor$ whenever $j > N(e)$ and $k > N(e)$. Observe that, since $f$ is summable, for any $\delta > 0$ there exists a finite set $E_f(\delta) \subset \Ed(\Gamma)$ of edges such that 
$$\sum_{e \in \Ed(\Gamma) \ssm E_f(\delta)} f(e) < \delta.$$
Let $N = \max\{N(e) \mid e \in E_f(\delta)\}$. Then $d_f(a_j, a_k) < \delta$ whenever $j > N$ and $k > N$. This means that $\{a_i\}$ is fundamental.

\smallskip

(2) If $\{d_f(a_i, b_i)\}$ tends to zero and $e \in \Ed(\Gamma)$, then there exists $N > 0$ such that $d_f(a_j, b_j) < f(e)$ whenever $j > N$. It follows that the set $\{i \in \N\, \mid\, e \in \Ed \lfloor a_i, b_i \rfloor\}$ consists of at most $N$ elements.

Conversely, assume that for every $e \in \Ed(\Gamma)$ there exists $N(e) > 0$ such that $e \notin \Ed \lfloor a_i, b_i \rfloor$ whenever $j > N(e)$. Let $E_f(\delta)$ be the finite set described in the proof of (1) and let $N = \max\{N(e) \mid e \in E_f(\delta)\}$. Then $d_f(a_j, b_j) < \delta$ whenever $j > N$. This means that $d_f(a_i, b_i) \to 0$ as $i \to \infty$, as required.

\smallskip

(3) follows from the previous assertion by setting $b_i = t$ for every $i$.  
\end{proof}

\begin{prop}
\label{pr:top-independent}
The topology of the metric space $(\overbar\Gamma, d_f)$ is independent of the choice of the summable function $f \colon \Ed(\Gamma) \to \R_+$. 
\end{prop}
\begin{proof}
It is well known (and easy to check) that the topology of a metrizable space is determined by its convergent sequences. By Lemma \ref{lem:1-convergence}(3), the set of convergent sequences in $(\overbar\Gamma, d_f)$ is independent of the choice of~$f$.
\end{proof}

For an edge $e \in \Ed(\Gamma)$ denote by $R_e$ the binary relation on the set of points of $\overbar\Gamma$ defined by 
$$(a,b) \in R_e \quad \Leftrightarrow \quad e \notin \Ed \lfloor a, b \rfloor.$$
It is easy to see that $R_e$ is an equivalence relation (transitivity follows from the assertion (f) of Lemma \ref{le:Y0}), which has precisely two equivalence classes (follows from the assertion (g) of Lemma \ref{le:Y0}).

\begin{prop}
\label{le:compact}
$\overbar\Gamma$ is compact in the topology induced by the metric $d_f$.
\end{prop}
\begin{proof} 
It is enough to show that every infinite sequence $\{x_i\}$ of points in $(\overbar\Gamma, d_f)$ has a convergent subsequence.

Let us first show that $\{x_i\}$ has a fundamental subsequence. Since $\Ed(\Gamma)$ is countable, $\{x_i\}$ has a subsequence $\{x_{i_j}\}$ such that for each $e \in \Ed(\Gamma)$ there exists $N(e) > 0$ such that $(x_{i_j}, x_{i_k}) \in R_e$ (which is equivalent to $e \notin \Ed \lfloor x_{i_j}, x_{i_k} \rfloor$) whenever $j > N(e)$ and $k > N(e)$. Then $\{x_{i_j}\}$ is fundamental by the assertion (1) of Lemma \ref{lem:1-convergence}.

Now, we show that every fundamental sequence $\{y_i\}$ in $(\overbar\Gamma, d_f)$ converges. Fix an arbitrary point $y \in \Gamma$ and consider the linear subtree
$$L = \bigcup_{k \in \N} \bigcap_{j > k} \lfloor y, y_j \rfloor.$$
There exists a unique point $t \in \overbar\Gamma$ such that $L = \lfloor y, t \rfloor$ (see Lemma \ref{le:all-ends}). We show that $\{y_i\}$ converges to~$t$. Assume that it does not. Then, by the assertion (3) of Lemma \ref{lem:1-convergence} there exists an edge $e$ and an infinite subsequence $\{y_{i_j}\}$ in $\{y_i\}$ such that $e \in \lfloor t, y_{i_j} \rfloor$ for all $j$. Since $\{y_i\}$ is fundamental, it then follows by the assertion (1) of Lemma \ref{lem:1-convergence} (and by the assertion (f) of Lemma \ref{le:Y0}) that $e \in \lfloor t, y_i \rfloor$ for all $i$ sufficiently large. 

In the case where $e \notin \lfloor y, t \rfloor$, we have $e \in \lfloor y, y_i \rfloor$ for all sufficiently large $i$ (we use the assertion (f) of Lemma \ref{le:Y0}), whence $e \in \lfloor y, t \rfloor$ (by definition of $L$ and $t$), a contradiction. 

In the case when $e \in \lfloor y, t \rfloor$ we have $e \notin \lfloor y, y_i \rfloor$ for all sufficiently large $i$ (we use the assertion (g) of Lemma \ref{le:Y0} and the fact that $e \in \lfloor t, y_i \rfloor$) whence $e \notin \lfloor y, t \rfloor$ (by definition of $L$ and $t$), a contradiction. 

This proves that $\{y_i\}$ converges to $t$ and the proposition is thus proved.
\end{proof}

\begin{remark}
It is obvious from the construction that $\Gamma$ is dense in $(\overbar\Gamma, d_f)$. Therefore, Proposition \ref{le:compact} implies that $\overbar\Gamma$ is the compactification of $\Gamma$ with respect to $d_f$.
\end{remark}

\subsection{Group action on $\overbar\Gamma$}
\label{subs:action}

Let $\Gamma$ be a countable $\Z^n$-tree and $\overbar\Gamma$ be the compactification of $\Gamma$ with respect to the metric $d_f$ constructed in Subsection \ref{subs:compact_general}.

Observe that an automorphism $\alpha \colon \Gamma \to \Gamma$ sends $\Gamma$-rays to $\Gamma$-rays, ends to ends, open cuts to open cuts, and therefore induces an automorphism $\overbar\alpha \colon \overbar\Gamma \to \overbar\Gamma$.

\begin{prop}
\label{prop:homeo}
Every automorphism of $\overbar\Gamma$ is a homeomorphism with respect to the metric $d_f$.
\end{prop}
\begin{proof} 
If $\alpha \colon \overbar\Gamma \to \overbar\Gamma$ is an automorphism, let $f_{\alpha} \colon Ed(\Gamma) \to \R_+$ be the function defined by 
$$f_{\alpha}(e) := f(\alpha^{-1}(e)).$$
Then $\alpha$ is an isometry between $(\overbar\Gamma, d_f)$ and $(\overbar\Gamma, d_{f_{\alpha}})$. Therefore, $\alpha$ is a homeomorphism because $d_f$ and $d_{f_{\alpha}}$ induce one and the same topology by Proposition \ref{pr:top-independent}.
\end{proof}

\begin{theorem}
\label{th:the-action}
If a group $G$ acts on $\Gamma$ by isometries, then the induced action of $G$ on $\overbar\Gamma$ is continuous with respect to the topology generated by the metric $d_f$. Moreover, if $G$ is non-abelian and acts on $\Gamma$ freely, without inversions, and so that $\Gamma$ has no proper $G$-invariant subtrees, then the orbit of every point in $\overbar\Gamma$ is infinite.
\end{theorem}
\begin{proof} The first part of the statement follows from Proposition \ref{prop:homeo}.

Now suppose there exists $v \in \overbar\Gamma$ such that $\{G \cdot v\}$ is finite. It follows that $|G : Stab_G(v)| < \infty$.

If $v \in \Gamma$, then $\Gamma$ must be spanned by $\{G \cdot v\}$ because otherwise $\Gamma$ contains a proper $G$-invariant subtree. But then, since the action is isometric, it follows that either there is a fixed point, or an inversion - a contradiction in either case. Thus we have $v \in \Ends(\Gamma) \cup Cuts(\Gamma)$. 

Suppose $v \in \Ends(\Gamma)$. By \cite[Lemma 3.1.9]{Chiswell:2001}, if $v$ is fixed by $g \in G$, then $v$ is an end of $Axis(g)$, which is a maximal linear $g$-invariant subtree of $\Gamma$. If $v$ is fixed by another element $h \in G$, then $v$ is an end of $Axis(h)$ too, so $Axis(g) = Axis(h)$, from which it follows that $[g,h] = 1$. Thus, $Stab_G(v)$ is abelian and $G$ is virtually abelian. But $G$ is $\Z^n$-free, so commutation is transitive and it follows that $G$ is abelian -- a contradiction with our assumption.

The case when $v \in \Cuts(\Gamma)$ is considered similarly.
\end{proof}

\begin{theorem}
\label{th:minimal}
Assume that a group $G$ acts on $\Gamma$ by isometries and $\Gamma$ has no proper $G$-invariant subtrees. Then the induced action of $G$ on the set of ends $\Ends(\Gamma)$ either has a global fixed point, or is topologically minimal (with respect to the topology induced by the metric $d_f$).
\end{theorem}
\begin{proof} 
Assume on the contrary that the action of $G$ on $\Ends(\Gamma)$ neither is topologically minimal, nor has a global fixed point. Then $\Ends(\Gamma)$ has a proper closed (in $\Ends(\Gamma)$) $G$-invariant subset $K$ containing at least two points. Set 
$$\Gamma_K = \bigcup_{a,b \in K} \lfloor a, b \rfloor.$$
We are going to obtain a contradiction by showing that $\Gamma_K$ is a proper $G$-invariant subtree of $\Gamma$.

First, observe that $\Gamma_K \neq \varnothing$. Indeed, since $K$ contains at least two points $a \neq b$, we have $\Gamma_K \supset \lfloor a, b \rfloor \neq \varnothing$. 

Next, observe that $\Gamma_K$ is a subtree in $\Gamma$. Indeed, if $x, y \in \Gamma_K$, then $x \in \lfloor a, b \rfloor$ and $y \in \lfloor c, d \rfloor$ for some $a, b, c, d \in K$. Interchanging $a$ with $b$ if necessary, we can assume that $b \neq c$. Then 
$$\lfloor a, b \rfloor \cap \lfloor b, c \rfloor \neq \varnothing \neq \lfloor b, c \rfloor \cap \lfloor c, d \rfloor,$$ 
so that the union 
$$N = \lfloor a, b \rfloor \cup \lfloor b, c \rfloor \cup \lfloor c, d \rfloor$$ 
is a subtree by \cite[Lemma 2.13]{Alperin_Bass:1987}. Therefore, we have $[x,y] \subset N \subset \Gamma_K$. It is thus shown that $[x,y] \subset \Gamma_K$ whenever $x,y \in \Gamma_K$, which means that $\Gamma_K$ is convex and hence (being non-empty) is a subtree of $\Gamma$.

Since $K$ is $G$-invariant, so is $\Gamma_K$. It remains to show that $\Gamma_K \neq \Gamma$.

Since $K$ is a proper closed subset in $\Ends(\Gamma)$, we can choose $p \in \Ends(\Gamma)$ and $\varepsilon > 0$ such that $\dist_{d_f}(p, K) > \varepsilon$. We show that there exists an edge $e$ in $\Gamma$ such that $\lfloor k, p \rfloor$ contains $e$ for each $k \in K$. Since $f$ is summable, there exists a finite set $E_f(\varepsilon) \subset \Ed(\Gamma)$ of edges such that 
$$\sum_{e \in \Ed(\Gamma) \ssm E_f(\varepsilon)} f(e) < \varepsilon.$$
We denote by $M$ the set of all the vertices of the edges in $E_f(\varepsilon)$. Observe that the intersection
$$L = \bigcap_{v \in M} [v, p \rangle$$
of $\Gamma$-rays is a $\Gamma$-ray because $M$ is finite. Let $e$ be an edge in $L$. If $k \in K$, then $\lfloor k, p \rfloor$ contains an edge of $E_f(\varepsilon)$ because $\dist_{d_f}(p, K) > \varepsilon$ by construction. Therefore, $\lfloor k, p \rfloor$ contains a point $v \in M$, whence we have
$$e \in \Ed(L) \subset \Ed[v,p) \subset \Ed \lfloor k, p \rfloor.$$
Now, we can show that $\lfloor x, p \rfloor$ contains $e$ for each $x\in \Gamma_K$. Indeed, if $x \in \Gamma_K$, then by the definition of $\Gamma_K$ there exist $a$ and $b$ in $K$ such that $x \in \lfloor a, b \rfloor$. We see that $\lfloor a, b \rfloor$ does not contain $e$ by the assertion (g) of Lemma \ref{le:Y0} because $\lfloor a, p \rfloor$ and $\lfloor b, p \rfloor$ both contain $e$ as proved above. Consequently, $\lfloor x, a \rfloor$ does not contain $e$ also, being a subset in $\lfloor a, b \rfloor$. Therefore, $\lfloor x, a \rfloor$ does not contain $e$ while $\lfloor a, p 
\rfloor$ contains $e$. Then $\lfloor x, p \rfloor$ contains $e$ because we have $\lfloor a, p \rfloor \subset \lfloor x, a \rfloor \cup \lfloor x, p \rfloor $ by the assertion (e) of Lemma \ref{le:Y0}.

Thus, we have shown that $\dist_{d_f}(\Gamma_K,p) \geqslant f(e) > 0$. Consequently, $\Gamma_K \neq \Gamma$. This completes the proof. 
\end{proof}

\subsection{Convergence in $\overbar\Gamma$}
\label{subs:ends}

In the space $\overbar\Gamma$ under investigation, there is a strong relationship between the convergence of points and the weak convergence of measures. For example, if $\Gamma$ is an ordinary $\Z^1$-tree ($n = 1$), $v_0$ a vertex in $\Gamma$, $z$ is an end in $\Ends(\Gamma)$, $\nu$ is a continuous Borel probability measure on $\overbar\Gamma$, and $\{\alpha_i\}$ is a sequence of automorphisms of $\Gamma$, then the following conditions are equivalent:
\begin{itemize}
\item[\rm(i)] the sequence $\{\alpha_i \cdot v_0\}$ converges to $z$,

\item[\rm(ii)] the sequence $\{\alpha_i \cdot \nu\}$ weakly converges to the Dirac measure $\delta_z$.
\end{itemize}
If $n > 1$, then (i) and (ii) are no longer equivalent. However, in the general case, the relationship mentioned above shows up when we impose appropriate restrictions on $z$ and/or $\{\alpha_i\}$. In fact, our proof of the main theorem is based on this relationship (see Corollary \ref{cor:mu-prox-2} and Proposition \ref {prop:measure-point}). In this subsection, we do some preliminary work that will allow us to prove Corollary \ref{cor:mu-prox-2} and Proposition \ref {prop:measure-point}.

\smallskip

Recall from Subsection \ref{subs:Z^n_trees} that $\Xi_{n-1}(\Gamma)$ is the set of all maximal $\Z^{n-1}$-subtrees of $\Gamma$.

\begin{defn}
\label{defn:nnorm}
Define the function $\ntn \colon \overbar\Gamma \times \overbar\Gamma \to [0, +\infty]$ as follows. Set $\ntn(a, a) = 0$ for each $a \in \overbar\Gamma$ and 
\begin{equation*}
\ntn(a,b) = \card\{S \in \Xi_{n-1}(\Gamma)\, \mid\, S \cap \lfloor a,b\rfloor \neq \varnothing\},
\end{equation*}
for distinct $a, b \in \overbar\Gamma$.
\end{defn}

That is, $\ntn(a,b)$ is the number of distinct maximal $\Z^{n-1}$-subtrees of $\Gamma$ intersecting with $\lfloor a, b\rfloor$. Observe that if $a$ and $b$ are distinct ends of $\Gamma$, then $\ntn(a,b)$ may be equal to $+\infty$.

\begin{lemma}
\label{lem:ntn}
The function 
$$\ntn \colon \overbar\Gamma \times \overbar\Gamma \to [0, +\infty]$$
is an \emph{extended metric} on $\overbar\Gamma$, that is, $\ntn$ satisfies the standard axioms of metric, but can attain the value $+\infty$. This extended metric is $G$-invariant. The restriction of this extended metric to $\Gamma$ is a metric.
\end{lemma}
\begin{proof}
Assertion (e) of Lemma \ref{le:Y0} implies that $\ntn$ satisfies the triangle inequality. The other required properties of $\ntn$ immediately follow from its definition.
\end{proof}

Recall that $\base$ is a designated point in $\Gamma$. By abuse of notation, for $p \in \overbar\Gamma$, we set $\ntn(p) = \ntn(\base, p)$, and for $g \in G$, we set
$$\ntn(g) = \ntn(g \cdot \base) = \ntn(\base, g \cdot \base).$$

\begin{lemma}
\label{lem:vrtx-convrg}  
Let $\{g_i\}$ be a sequence in $G$, let $v_0$ be a point in $\Gamma$, and let $z$ be a point in $\overbar\Gamma$. Assume that the sequence $\{g_i \cdot v_0\}$ converges to $z$ and that the sequence $\{\ntn(g_i)\}$ tends to infinity as $i \to \infty$. Then all the sequences $\{g_i \cdot v\}$ with $v \in \Gamma$ converge to $z$.
\end{lemma}
\begin{proof} 
Since the action of $G$ on $\overbar\Gamma$ is isometric with respect to the extended metric $\ntn$, the assumption that $\{\ntn(g_i)\} = \{\ntn(g_i \cdot \base, \base)\}$ tends to infinity implies that $\{\ntn(g_i\cdot v)\}$ tends to infinity for each $v \in \Gamma$. Indeed, we observe that $\ntn(g_i \cdot \base, g_i \cdot v) = \ntn(\base, v) = \ntn(v)$ is finite, while the triangle inequality (see Lemma \ref{lem:ntn}) implies that 
$$\ntn(g_i \cdot v) = \ntn(g_i\cdot v, \base) \geqslant \ntn(g_i \cdot \base, \base) - \ntn(g_i \cdot \base, g_i \cdot v) =\ntn(g_i) - \ntn(v).$$
In particular, $\{\ntn(g_i \cdot v_0)\}$ tends to infinity.

Now, we take $v \in \Gamma$, put $k := \ntn(v_0, v)$, and observe that for each $i \in \N$ the geodesic segment $[g_i \cdot v_0, g_i \cdot v]$ is contained in the $\ntn$-ball of $\ntn$-radius $k$ centered at $g_i \cdot v_0$. Since  $\{\ntn(g_i \cdot v_0)\}$ tends to infinity, this implies that the $\ntn$-distance between $\base$ and $[g_i \cdot v_0,  g_i \cdot v]$ tends to infinity. Therefore, for each $e \in \Ed(\Gamma)$ the set $\{i \in \N\, \mid\, e \in \Ed[g_i \cdot v_0, g_i \cdot v]\}$ is at most finite. The result then follows by the assertion (2) of Lemma \ref{lem:1-convergence}.
\end{proof}

\begin{lemma}
\label{le:nnorm-convrg}
Let $\{a_i\}$ and $\{b_i\}$ be two sequences of points in $\overbar\Gamma$. Suppose that $\ntn(Y(\base,a_i,b_i)) \to \infty$ as $i \to \infty$. Then $d_{\overbar\Gamma}(a_i,b_i) \to 0$ as $i \to \infty$. In particular, $\{a_i\}$ converges to a point $\omega \in \overbar\Gamma$ if and only if $\{b_i\}$ converges to $\omega$.
\end{lemma}
\begin{proof} 
Observe that if $x \in \overbar\Gamma$ and $y \in \overbar\Gamma$ belongs to $\lfloor \base, x \rfloor$ (as a point or an open cut), then $\ntn(y) \leqslant \ntn(z)$ for every $z \in \lfloor y, x \rfloor$ (follows from definition of $\ntn$). 

Next, from the assertion (h) of Lemma \ref{le:Y0} it follows that for any $a, b \in \overbar \Gamma$ such that $\lfloor a, b \rfloor \neq \varnothing$ we have 
$$\lfloor \base, a \rfloor \cap \lfloor \base, b \rfloor \subset \lfloor \base, z \rfloor$$
for any $z \in \lfloor a, b \rfloor$. Hence, it follows that
$$\min_{z \in \lfloor a, b \rfloor} \ntn(z) = \ntn(Y(\base, a, b)).$$
Therefore, we have $\min_{z \in \lfloor a, b \rfloor} \ntn(z) = \ntn(Y(\base, a_i, b_i))$ for all $i \in \N$. Since we assume $\ntn(Y(\base, a_i, b_i)) \to \infty$ as $i \to \infty$, it follows that for each $v \in \Gamma$ the set $\{i \in \N\, \mid\, v \in [a_i, b_i]\}$ is at most finite (because we have $v \notin [a_i,b_i]$ whenever $\min_{z \in \lfloor a, b \rfloor} \ntn(z) > \ntn(v)$). Then the result follows from the assertion (2) of Lemma \ref{lem:1-convergence}.
\end{proof}

\begin{lemma}
\label{lem:open-product}
For every vertex $v$ of a $\Z^n$-tree $\Gamma$, the set 
$$I_v = \{(a,b) \in \Ends(\Gamma) \times \Ends(\Gamma) \mid v \in \lfloor a, b \rfloor \}$$
is open in $\Ends(\Gamma) \times \Ends(\Gamma)$ with respect to the product topology induced by the metric $d_f$.
\end{lemma}
\begin{proof} 
Let $(a,b) \in I_v$ and let $p \in \Ends(\Gamma)$ be such that $d_f(b, p) < d_f(v, b)$. Then $v \notin \lfloor p, b \rfloor$. However, we have $\lfloor a, b \rfloor \subseteq \lfloor a, p \rfloor \cup \lfloor p, b \rfloor$ by the assertion (e) of Lemma \ref{le:Y0}. This implies that $v \in \lfloor a, p \rfloor$, that is, $(a,p) \in I_v$.

By repeating the same argument we see that $(r, p) \in I_v$ whenever $d_f(b, p) < d_f(v,b)$ and $d_f(a,r) < d_f(a,v)$. This obviously implies the lemma.
\end{proof}

\begin{remark}
It can be shown that $I_v$ from Lemma \ref{lem:open-product} is also closed.
\end{remark}

\section{Proof of Theorem \ref{th:Poisson}}
\label{sec:poisson}

In this section, we prove Theorem \ref{th:Poisson} according to the scheme outlined in the introduction. For the rest of this section, we fix $n \in \N$ and a countable non-abelian $\Z^n$-free group $G$ acting minimally, freely, and without inversion on a $\Z^n$-tree $\Gamma$. Also, we let $(\overline\Gamma, d_f)$ denote the compact metric space constructed in Subection \ref{subs:compact_general} for $\Gamma$. Recall that $\Ends(\Gamma)$ denote the set of open ends of $\Gamma$. We also fix a non-degenerate probability measure $\mu$ on $G$.

\subsection{$\mu$-proximality of $\overbar\Gamma$ and $\Ends(\Gamma)$}
\label{subs:4.1}

\begin{prop}
\label{prop:where-stationary-measures-lives}
If $\nu$ is a $\mu$-stationary measure on $\overbar\Gamma$ then
\begin{itemize}
\item[\rm(i)] $\nu$ is continuous,

\item[\rm(ii)] $\nu(\Ends(\Gamma)) = 1$,

\item[\rm(iii)] $\nu(E) > 0$ for every non-empty open set $E \subset \overbar\Gamma$ with $E \cap \Ends(\Gamma) \neq \varnothing$.
\end{itemize}
\end{prop}
\begin{proof} 
(i) Continuity of $\nu$ follows from Lemma \ref{lem:KaiMas} because the action of $G$ on $\overbar\Gamma$ has no finite orbits (see Theorem \ref{th:the-action}).

\smallskip

(ii) The set $\Gamma$ is countable by assumption. Therefore, $\Cuts(\Gamma)$ is countable. Then, since $\nu$ is continuous, we have $\nu(\Gamma) = 0$ and $\nu(\Cuts(\Gamma)) = 0$. Hence, $\nu(\Ends(\Gamma)) = 1$.

\smallskip

(iii) From Theorem \ref{th:minimal} it follows that either $G$ acts on $\Ends(\Gamma)$ minimally, or there is a global fixed point. If all elements of $G$ fix $e \in \Ends(\Gamma)$, then $G$ is abelian and we have a contradiction with the set of assumptions on $G$. Hence, $G$ acts on $\Ends(\Gamma)$ minimally and now (iii) follows from (ii) by Lemma \ref{lem:minimal-stationary}.
\end{proof}

Define $\mathfrak{S}$ to be the subset in $G^{\N_0}$ consisting of all sequences $\{g_i\}$ with the following three properties:
\begin{itemize}
\item[(i)] $\{g_i \cdot \base\}$ converges in $(\overbar\Gamma, d_f)$,

\item[(ii)] $\{g_i^{-1} \cdot \base\}$ converges in $(\overbar\Gamma, d_f)$,

\item[(iii)] $\{\ntn(g_i)\}$ tends to infinity as $i \to \infty$ (recall Definition \ref{defn:nnorm}).
\end{itemize}

\begin{prop}
\label{prop:mu-prox-1}
A.\,e.~path $\{\tau_i\}$ of the $\mu$-random walk contains a subsequence $\{\tau_{i_j}\} \in \mathfrak{S}$.
\end{prop}
\begin{proof} 
First, we prove that a.\,e.~path $\{\tau_i\}$ of the $\mu$-random walk contains a subsequence with the property (iii) from the definition of $\mathfrak{S}$. 

Let $d \in \N$. Since $G$ acts on $\Gamma$ minimally, it follows that there exists an element $g_{2d} \in G$ such that $\ntn(g_{2d}) > 2d$. Since $\mu$ is a non-degenerate measure, there exists $k \in \N$ such that $\mu^{*k}(g_{2d}) > 0$. It follows that for a.\,e.~path $\{\tau_i\}$ of the $\mu$-random walk there exists $m \in \N$ such that $\tau_{m+k} = \tau_m g_{2d}$ (to see this, use ergodicity of the Bernoully shift in the space of sequences of $\mu^{*k}$-equidistributed random variables $\{\tau_{jk}^{-1} \tau_{jk + k}\}_{j \in \N_0}$, or observe that the measure $P_\mu$ of the set of all paths of the $\mu$-random walk with $\tau_{jk + k} \neq \tau_{jk} g_{2d}$ for each $j \in \N$ equals $\lim_{s \to \infty}(1 - \delta)^s = 0$, where $\delta = \mu^{*k}(g_{2d}) > 0$). Then
$$\ntn(\tau_m^{-1} \tau_{m+k}) = \ntn(g_{2d}) > 2d.$$
Now, we apply the triangle inequality (see Lemma \ref{lem:ntn}) and deduce that
$$\max\{\ntn(\tau_{m+k}), \ntn(\tau_m)\} > d.$$
We have thus shown that for any $d \in \N$, the path $\{\tau_i\}$ a.\,s.~has an element $\tau_i$ such that $\ntn(\tau_i) > d$. It obviously follows that $\{\tau_i\}$ a.\,s.~has a subsequence $\{\tau'_i\}$ such that $\{\ntn(\tau'_i)\}$ tends to infinity (the property (iii)).

\smallskip

Since $(\overbar\Gamma,d_f)$ is compact, each infinite sequence in $G$ (and $\{\tau'_i\}$ in particular) contains a subsequence having both properties (i) and (ii).
\end{proof}

\begin{lemma}
\label{lem:mu-prox-2}
Let $\{g_i\}$ be a sequence from $\mathfrak{S}$. Let $\omega$ and $\check\omega$ be points from $\overbar\Gamma$ that the sequences $\{g_i \cdot \base\}$ and $\{g_i^{-1} \cdot \base\}$ converge to. Then, for each $p \in \overbar\Gamma \ssm \check\omega$, the sequence $\{g_i \cdot p\}$ converges to $\omega$.
\end{lemma}
\begin{proof} In order to prove the lemma, we will show that for every $p \in \overbar\Gamma \ssm \check\omega$ we have
\begin{equation}
\label{eq:234}
\ntn(p_i) \xrightarrow[i \to \infty]{} \infty,
\end{equation}
where $p_i = Y(\base, g_i \cdot \base, g_i \cdot p)$, and then apply Lemma \ref{le:nnorm-convrg}.

\smallskip

Observe that for each $i$ we have $g_i^{-1} \cdot p_i = Y(\base, g_i^{-1} \cdot \base, p)$ because
$$g_i^{-1} \cdot p_i = g_i^{-1} \cdot Y(\base, g_i \cdot \base, g_i \cdot p) = Y(g_i^{-1} \cdot \base, \base, p) = Y(\base, g_i^{-1} \cdot \base, p).$$

Let us prove that the sequence $\{\ntn(Y(\base, g_i^{-1} \cdot \base, p))\}$ is bounded. Assume that it is not. Then it has a subsequence $\{\ntn(Y(\base, g_{i_j}^{-1} \cdot \base, p))\}$ which tends to infinity. Then, by Lemma \ref{le:nnorm-convrg}, $\{g_{i_j}^{-1} \cdot \base\}$ converges to $p$, which contradicts the assumption that $g_i^{-1} \cdot \base \to \check\omega \neq p$. Thus, $\{\ntn(Y(\base, g_i^{-1} \cdot \base, p))\}$ is bounded, that is, there exists $N \in \N$ such that
$$\ntn(Y(\base, g_i^{-1} \cdot \base, p)) < N \quad \text{for each} \quad i \in \N.$$
Therefore, for each $i \in \N$ we have
$$\ntn(g_i \cdot \base, p_i) = \ntn(\base, g_i^{-1} \cdot p_i) \stackrel{\rm def}{=} \ntn(g_i^{-1} \cdot p_i) = \ntn(Y(\base, g_i^{-1} \cdot \base, p)) < N.$$
Now, $\{g_i\} \in \mathfrak{S}$, so 
$$\ntn(\base, g_i \cdot \base) \stackrel{\rm def}{=} \ntn(g_i \cdot \base) \stackrel{\rm def}{=} \ntn(g_i)$$
tends to $\infty$ (the property (iii)), while we have 
$$\ntn(\base, p_i) + \ntn(p_i, g_i \cdot \base) \ge \ntn(\base, g_i \cdot \base)$$
by the triangle inequality (see Lemma~\ref{lem:ntn}), whence it follows that $\ntn(p_i) \to \infty$. Finally, by Lemma \ref{le:nnorm-convrg}, this implies that $\{g_i \cdot p\}$ converges to $\omega$ since $\{g_i \cdot \base\}$ does.
\end{proof}

Lemma \ref{lem:mu-prox-2} obviously implies the following corollary.

\begin{corollary}
\label{cor:mu-prox-2}
Let $\{g_i\}$ be a sequence from $\mathfrak{S}$ with $\{g_i \cdot \base\}$ converging to a point $\omega \in \overbar\Gamma$. Then, for any continuous Borel probability measure $\lambda$ on $\overbar\Gamma$, the sequence $\{g_i \cdot \lambda\}$ converges to the Dirac measure $\delta_\omega$.
\end{corollary}

\begin{corollary}
\label{cor:each-stat-is-muboundary}
If $\nu$ is a $\mu$-stationary measure on $\overbar\Gamma$, then $(\overbar\Gamma, \nu)$ is a $\mu$-boundary of $(G, \mu)$ in the sense of Furstenberg. In other words, the $G$-space $\overbar\Gamma$ is mean-proximal.
\end{corollary}
\begin{proof} 
Let $\{\tau_i\}$ be a path of the $\mu$-random walk. Then by Theorem~\ref{thm:alw-conv}, the sequence $\{\tau_i \cdot \nu\}$ a.\,s.~converges to some limit. By Proposition~\ref{prop:mu-prox-1}, $\{\tau_i\}$ a.\,s.~contains a subsequence $\{\tau_{i_j}\} \in \mathfrak{S}$. By Corollary~\ref{cor:mu-prox-2}, the subsequence $\{\tau_{i_j} \cdot \nu\}$ converges to a Dirac measure (since $\nu$ is continuous by Proposition \ref{prop:where-stationary-measures-lives}). Therefore, the limit of $\{\tau_i \cdot \nu\}$ is a.\,s.~a Dirac measure. This means by definition that $(\overbar\Gamma, \nu)$ is a $\mu$-boundary of $(G,\mu)$ in the sense of Furstenberg. Since $\mu$ is an arbitrary non-degenerate measure, it follows by Theorem~\ref{thm-def:mean-proximal} that the $G$-space $\overbar\Gamma$ is mean-proximal.
\end{proof}

Now, we can draw the main result of this subsection from the results above.

\begin{prop}
\label{prop:compact_proximal}
There exists a unique $\mu$-stationary measure $\nu_\mu$ on $\overbar\Gamma$. This measure is continuous and concentrated on~$\Ends(\Gamma)$. The measure space $(\overbar\Gamma, \nu_\mu)$ is a $\mu$-boundary of $G$ in the sense of Furstenberg.
\end{prop}
\begin{proof}  
Observe that uniqueness of a $\mu$-stationary measure $\nu_\mu$ follows from Corollary~\ref{cor:each-stat-is-muboundary} by Theorem~\ref{thm-def:mean-proximal} (the last one is applicable, because $\overbar\Gamma$ is compact and metrizable). The rest follows from Proposition~\ref{prop:where-stationary-measures-lives} and Corollary~\ref{cor:each-stat-is-muboundary}.
\end{proof}

Proposition~\ref{prop:compact_proximal} yields the following corollary (see Assertion~\ref{asser:bnd} and Corollary~\ref{rem:quotients} concerning relationship between $\mu$-boundaries in the sense of Furstenberg and in the sense of Kaimanovich).

\begin{corollary}
\label{prop:ends_proximal}
There exists a unique $\mu$-stationary measure $\nu_\mu$ on $\Ends(\Gamma)$. This measure is continuous. The measure space $(\Ends(\Gamma), \nu_\mu)$ is a $\mu$-boundary of $G$ in the sense of Kaimanovich.
\end{corollary}

\subsection{Stability of paths in $\Z^n$-free groups}
\label{subs:4.2}

\begin{prop}
\label{prop:measure-point}
Let $\{g_i\}$ be a sequence in~$G$, and let $z$ be a point in~$\overbar\Gamma$. Assume that the sequence $\{g_i \cdot \nu_\mu\}$ converges to the Dirac measure $\delta_z$. Then all the sequences $\{g_i \cdot v\}$ with $v \in \Gamma$ converge to~$z$.
\end{prop}
\begin{proof} 
Let $v \in \Gamma$. Since $\Gamma$ is countable and admits a minimal action of a group, it follows that there exists a pair of ends $\omega_1, \omega_2 \in \Ends(\Gamma)$ such that $v \in \langle \omega_1, \omega_2 \rangle$ (indeed, otherwise $\Gamma$ has closed ends and the subset of vertices that are not closed ends forms a proper invariant subtree, which contradicts the minimality assumption).

Lemma~\ref{lem:open-product} implies that there exist non-empty open subsets $A$ and $B$ in $\Ends(\Gamma)$ containing respectively $\omega_1$ and $\omega_2$ such that $v \in \langle a, b \rangle$ whenever $a \in A$, $b \in B$. By assertion~(iii) of Proposition~\ref{prop:where-stationary-measures-lives} we have $\nu_\mu(A) > 0$ and $\nu_\mu(B) > 0$. Therefore, since $\{g_i \cdot \nu_\mu\}$ converges (weakly) to~$\delta_z$, it follows that for every $r > 0$ there exists $N > 0$ such that $d_f(z, g_i \cdot a_i) < r$ and $d_f(z, g_i \cdot b_i) < r$ for some $a_i \in A$ and $b_i \in B$ whenever $i > N$.

Note that the balls of the metric~$d_f$ are {\em geodesically convex\/}. In particular, the conditions 
$$g_i \cdot v \in \langle g_i \cdot a_i, g_i \cdot b_i \rangle, \quad d_f(z, g_i \cdot a_i) < r, \quad \text{and} \quad d_f(z, g_i \cdot b_i) < r$$ 
imply that $d_f(z, g_i \cdot v) < r$. Indeed, we have 
$$g_i \cdot v \in \langle g_i \cdot a_i, g_i \cdot b_i \rangle = \lfloor g_i \cdot a_i, g_i \cdot b_i \rfloor \subset \lfloor z, g_i \cdot a_i \rfloor \cup \lfloor z, g_i \cdot b_i \rfloor$$ 
by Lemma~\ref{le:Y0}, so that 
$$d_f(z, g_i \cdot v) \leqslant d_f(z, g_i \cdot a_i) < r \quad \text{if} \quad g_i \cdot v \in \lfloor z, g_i \cdot a_i \rfloor,$$ 
$$d_f(z, g_i \cdot v) \leqslant d_f(z, g_i \cdot b_i) < r \quad \text{if} \quad g_i \cdot v \in \lfloor z, g_i \cdot b_i \rfloor.$$

It is thus shown that for any $r > 0$ there exists $N$ such that $d_f(z, g_i \cdot v) < r$ whenever $i > N$, that is, the sequence $\{g_i \cdot v\}$ converges to $z$.
\end{proof}

\begin{corollary}
\label{cor:path_stabilization}
For almost every path $\tau = \{\tau_i\}$ of the $\mu$-random walk, all the sequences $\{\tau_i \cdot v\}$, where $v \in \Gamma$, converge to a random end $\omega(\tau) \in \Ends(\Gamma)$; the distribution of the limits $\omega(\tau)$ is given by the $\mu$-stationary measure~$\nu_\mu$.
\end{corollary}
\begin{proof}
Proposition~\ref{prop:compact_proximal} says that the measure space $(\overbar\Gamma, \nu_\mu)$ is a $\mu$-boundary of $G$ in the sense of Furstenberg. This means by definition that, for almost every path $\tau = \{\tau_i\}$ of the $\mu$-random walk, the sequence $\{\tau_i \cdot \nu_\mu\}$ converges weakly to a random Dirac measure $\delta_z$, where $z = z(\tau) \in \overbar\Gamma$. 

Then Proposition~\ref{prop:measure-point} implies that all the sequences $\{\tau_i \cdot v\}$, where $v \in \Gamma$, converge to~$z(\tau)$. By Assertion~\ref{asser:bnd}, the distribution of the limits $z(\tau)$ is given by the measure~$\nu_\mu$. Since $\nu_\mu(\Ends(\Gamma)) = 1$ (see Proposition~\ref{prop:compact_proximal}), this implies the required assertion.
\end{proof}

\subsection{Maximality of the boundary $\Ends(\Gamma)$}
\label{subs:4.3}

According to Corollary \ref{prop:ends_proximal}, the pair $(\Ends(\Gamma), \nu_\mu)$ is a $\mu$-boundary of $G$. Now we would like to show that it is a maximal $\mu$-boundary. 

\smallskip

Let $X$ be a measurable $G$-space with a (quasi-invariant) measure $\nu$. Recall that the action of $G$ on $X$ is called {\em ergodic} if $\nu(Y) \in \{0,1\}$ for each $G$-invariant measurable subset $Y \subset X$.

\begin{lemma}[\cite{Kaimanovich:2000}]
\label{lem:ergo}
Let $G$ be a countable group, $\mu \in \PP(G)$ and $\check\mu \in \PP(G)$ the {\em reflected measure} defined by $\check\mu(g) = \mu(g^{-1})$. Let $(M_+, \nu_+)$ and $(M_-, \nu_-)$ be respectively a $\mu$-boundary and a $\check\mu$-boundary of $G$. Then the action of $G$ on the product $(M_- \times M_+, \nu_- \times \nu_+)$ is ergodic.
\end{lemma}

Observe that in \cite{Kaimanovich:2000}, Lemma \ref{lem:ergo} is stated and proved in the case when $(M_+, \nu_+)$ and $(M_-, \nu_-)$ are maximal, respectively, $\mu$- and $\check\mu$-boundaries. But the proof works verbatim for the general case.

\begin{prop}
\label{prop:maximality}
Let $G$ be a finitely generated non-abelian $\Z^n$-free group acting on a $\Z^n$-tree $\Gamma$ for some $n \in \N$ so that $\Gamma$ has no proper $G$-invariant subtrees. Let $\mu$ be a non-degenerate probability measure on $G$ and $\nu_\mu$ the unique $\mu$-stationary measure on the space $\Ends(\Gamma)$ of open ends of~$\Gamma$. If $\mu$ has finite first moment with respect to a finite word metric on $G$, then the measure space $(\Ends(\Gamma), \nu_\mu)$ is a maximal $\mu$-boundary of $G$.
\end{prop}
\begin{proof} In order to prove the theorem, we are going to construct a map 
$$S \colon \Ends(\Gamma) \times \Ends(\Gamma) \to 2^G$$ 
and then show that it satisfies all the requirements of the Strip Criterion (see Corollary \ref{cor:strip-criterion} and Remark \ref{rem:empty-strips}).

\smallskip

Let $a, b \in \Ends(\Gamma)$. There exists a unique bi-infinite geodesic $\langle a, b \rangle$ in $\Gamma$ from $a$ to $b$. We define
$$S(a,b) = \{g \in G \mid g \cdot \base \in \langle a, b \rangle\}.$$
Obviously, the map $S \colon \Ends(\Gamma) \times \Ends(\Gamma) \to 2^G$ is $G$-equivariant.

\smallskip

{\bf Claim 0.} {\em The map $S$ is measurable in the sense that $S^{-1}(M)$ is measurable for any $M \subset G$.}

\smallskip 

For a vertex $v \in \Gamma$, define
$$\label{Omega-def} \Omega_v = \{(a, b) \in \Ends(\Gamma) \times \Ends(\Gamma) \mid v \in \langle a, b \rangle\}$$
Our definitions immediately imply that
\begin{equation}
\label{Omega-eq}
S(a,b) = \{g \in G \mid (a, b) \in \Omega_{g \cdot \base} = g \cdot \Omega_{\base}\}
\end{equation}
so that for each $g \in G$ we have
$$S^{-1}(g) = \Omega_{g^{-1} \cdot \base}$$
By Lemma \ref{lem:open-product}, for all $v \in \Gamma$ the sets $\Omega_v$ are open and hence Borel measurable. In particular, the set $S^{-1}(g) = \Omega_{g^{-1} \cdot \base}$ is Borel measurable for every $g \in G$. Hence, {\bf Claim 0} follows by the countability of~$G$.

\smallskip

{\bf Claim 1.} {\em Let $\nu_+$ be a $\mu$-stationary measure on $\Ends(\Gamma)$ and $\nu_-$ be a $\check\mu$-stationary measure on $\Ends(\Gamma)$, where $\check\mu$ is the reflected measure of $\mu$ defined by $\check\mu(g) = \mu(g^{-1})$. Then, for $(\nu_- \times \nu_+)$-a.\,e.~pair $(a,b) \in \Ends(\Gamma) \times \Ends(\Gamma)$, the set $S(a,b)$ is non-empty.}

\smallskip

It follows from~\eqref{Omega-eq} that
\begin{equation}
\label{Omega-eq-2}
S(a,b) = \{g \in G \mid g^{-1} \cdot (a, b) \in \Omega_\base\}.
\end{equation}

Since $\Gamma$ is countable and admits a minimal action of a group (while $\Ends(\Gamma) \neq \varnothing$), it follows that $\Omega_\base$ is not empty (indeed, otherwise $\Gamma$ would have closed ends and the subset of vertices that are not closed ends would form a proper invariant subtree, which contradicts the minimality assumption).

Lemma \ref{lem:open-product} says that $\Omega_\base$ is open with respect to the product topology. Therefore, since $\Omega_\base$ is not empty, there exist non-empty open sets $A, B \subset \Ends(\Gamma)$ with $A \times B \subset \Omega_\base$. Applying the assertion (iii) of Proposition \ref{prop:where-stationary-measures-lives} to the measures $\nu_-$ and $\nu_+$, we see that $(\nu_- \times \nu_+)(\Omega_\base) > 0$.

Recall that the action of $G$ on the space $(\Ends(\Gamma) \times \Ends(\Gamma), \nu_- \times \nu_+)$ is ergodic (Lemma \ref{lem:ergo}). Consequently, since $(\nu_- \times \nu_+)(\Omega_\base) > 0$, it follows by~\eqref{Omega-eq-2} that for $(\nu_- \times \nu_+)$-a.\,e.~pair $(a,b) \in \Ends(\Gamma) \times \Ends(\Gamma)$, the set $S(a, b)$ is non-empty and {\bf Claim 1} follows.

\smallskip

Now, recall that by assumption $G$ is a finitely generated non-abelian group acting freely and without inversions on $\Gamma$, and $\Gamma$ has no proper $G$-invariant subtrees. Let $Z$ be a finite generating set of $G$ and let $| \cdot |_G$ be the word metric on $G$ with respect to $Z$.

To finish the proof of the theorem we are going to prove the following claim.

\smallskip

{\bf Claim 2.} {\em For any $m \in \{1, \ldots, n\}$ there exists $C(m) \in \mathbb{N}$ such that for any maximal $\Z^m$-subtree $T$ of $\Gamma$, and any distinct $a, b \in \Ends(T)$ we have
$$\card\left(\{g \in G \mid g \cdot \base \in \langle a, b \rangle\} \cap B_G(k)\right) \leqslant C(m) k^m,$$
where $B_G(k) = \{g \in G \mid |g|_G \leqslant k\}$.}

\smallskip

Fix a maximal $\Z^m$-subtree $T$ and assume that there exists at least one $g \in G$ such that $g \cdot \base \in \langle a, b \rangle$, otherwise the claim holds trivially for $T$. Define $H \colon = Stab_G(T)$. Since $G$ is finitely generated, $H$ is also finitely generated, and it follows that $H$ is quasi-isometrically embedded into $G$ (see Subsection \ref{subs:Z^n_trees}). Let $R$ be a finite generating set of $H$ and let $| \cdot |_H$ be the word metric on $H$ with respect to $R$.

Denote by $T_{\base,m}$ the maximal $\Z^m$-subtree of $\Gamma$ containing $\base$. In order to prove the claim we use induction on $m$. 

\smallskip

If $m = 1$, then $T$ is a $\Z$-tree and $\langle a, b \rangle$ is isometric to $\Z$. Consider two cases.

\begin{enumerate}
\item[(1.1)] Assume that $T = T_{\base,1}$, that is, $\base \in T$. If $g \in G$ is such that $g \cdot \base \in \langle a, b \rangle$, then $g \in H$ and we have 
$$\{g \in G \mid g \cdot \base \in \langle a, b \rangle\} = \{g \in H \mid g \cdot \base \in \langle a, b \rangle\}.$$
Observe that $H$ is a finitely generated free group which quasi-isometrically embeds into $T$ by means of the map $h \to h \cdot \base$ for any $h \in H$. 

Now, let $g \in H$ be such that $g \in B_G(k)$, that is, $|g|_G \leqslant k$. Since $H$ quasi-isometrically embeds into $G$, it follows that $|g|_H \leqslant C_1 k + C_2$ for some constants $C_1, C_2 \in \N$ which depend only on $H$ and $G$. Then, it follows that 
$$d_T(\base, g \cdot \base) \leqslant C_3 (C_1 k + C_2) + C_4 \leqslant C' k,$$
where $d_T$ is the distance on $T$ inherited from $\Gamma$ and $C'$ depends only on $H$ and $T$. But now observe that since $\langle a, b \rangle$ is a linear subtree of $T$, the number of points on $\langle a, b \rangle$ at distance not greater than $C' k$ from $\base$ is bounded by $2 C' k + 1$. At the same time, the action of $H$ on $T$ is free, so, the number of elements $g \in H$ such that $g \in B_G(k)$ and $g \cdot \base \in \langle a, b \rangle$ cannot be greater than $2 C' k + 1$. It means that the required inequality 
$$\card\left(\{g \in G \mid g \cdot \base \in \langle a, b \rangle\} \cap B_G(k)\right) \leqslant C(T_{\base,1}) k$$
holds for $C(T_{\base,1}) = 2 C' + 1$.

\item[(1.2)] Assume that $T \neq T_{\base,1}$. Observe that if
$$g_0 \in \{g \in G \mid g \cdot \base \in \langle a, b \rangle\} \cap B_G(k),$$
then $g_0^{-1} \cdot a$ and $g_0^{-1} \cdot b$ are distinct ends of $T_{\base,1}$ and we have
$$\card\left(\{g \in G \mid g \cdot \base \in \langle a, b \rangle\} \cap B_G(k)\right)$$
$$\leqslant \card\left(\{g \in G \mid g \cdot \base \in \langle g_0^{-1} \cdot a, g_0^{-1} \cdot b \rangle\} \cap B_G(2k)\right).$$
Indeed, for every $g_1 \in \{g \in G \mid g \cdot \base \in \langle a, b \rangle\} \cap B_G(k)$, the element $g_0^{-1} g_1$ belongs to $B_G(2k)$ since 
$$|g_0^{-1} g_1|_G \leqslant |g_0|_G + |g_1|_G \leqslant 2k,$$
and 
$$(g_0^{-1} g_1) \cdot \base \in \langle g_0^{-1} \cdot a, g_0^{-1} \cdot b \rangle.$$
In other words, $g_0^{-1} g_1 \in \{g \in G \mid g \cdot \base \in \langle g_0^{-1} \cdot a, g_0^{-1} \cdot b \rangle\} \cap B_G(2k)$ and the required inequality follows. Finally, from (1.1) we have that
$$\card\left(\{g \in G \mid g \cdot \base \in \langle g_0^{-1} \cdot a, g_0^{-1} \cdot b \rangle\} \cap B_G(2k)\right) \leqslant C(T_{\base,1}) \cdot (2k).$$
That is, $C(T) = 2 C(T_{\base,1})$.

\end{enumerate}

Hence, the constant $C(1)$, which works for every maximal $\Z$-subtree of $\Gamma$ can be taken to be $2 C(T_{\base,1})$. This concludes the proof of the claim in the case when $m = 1$.

\smallskip

Suppose $m > 1$ and assume that {\bf Claim 2} holds for $m - 1$.

Recall from Subsection \ref{subs:Z^n_trees} that one can define the metric $d_{m-1}$ on the set $\Xi_{m-1}(T)$ of all maximal $\Z^{m-1}$-subtrees of $T$ and $d_{m-1}$ takes values in $\Z$. Consider two cases.

\begin{enumerate}
\item[(m.1)] Assume that $T = T_{\base,m}$, that is, $\base \in T$. In particular, $T$ contains $T_{\base,m-1}$.

Observe that $\langle a, b \rangle$ may have non-empty intersection with at most countably many $\Z^{m-1}$-subtrees of $T$, which we can enumerate by integers
$$\ldots, T_{-r}, \ldots, T_{-1}, T_0, T_1, \ldots, T_r, \ldots,$$
where $T_0$ is the ``closest'' one to $T_{\base,m-1}$ with respect to the metric $d_{m-1}$. 

There exists a constant $M > 0$ such that for every $h \in \{g \in G \mid g \cdot \base \in \langle a, b \rangle\} \cap B_G(k)$ we have 
$$d_{m-1}(T_{\base,m-1}, h \cdot T_{\base,m-1}) \leqslant M k.$$ 
Indeed, since $h \cdot \base \in \langle a, b \rangle$, it follows that $h \in H$. Next, since $H$ quasi-isometrically embeds into $G$, from $|h|_G \leqslant k$ it follows that $|h|_H \leqslant C_1 k + C_2$ for some constants $C_1, C_2 \in \N$ which depend only on $H$ and $G$. Now, $d_{m-1}$ satisfies the triangle inequality, so if
$$L = \max_{h_i \in R} d_{m-1}(T_{\base,m-1},  h_i \cdot T_{\base,m-1}),$$
then 
$$d_{m-1}(T_{\base,m-1}, h \cdot T_{\base,m-1}) \leqslant L \cdot (C_1 k + C_2)$$ 
and existence of the required $M$ follows.

Let
$${\mathcal T}_k \colon = \{S \in \{T_i\} \mid d_{m-1}(S, T_{\base,m-1}) \leqslant M k\}.$$
Obviously, $\card({\mathcal T}_k) \leqslant 2 M k + 1$ and for every $h \in \{g \in G \mid g \cdot \base \in \langle a, b \rangle\} \cap B_G(k)$ there exists $S \in {\mathcal T}_k$ such that $h \cdot \base \in \langle a, b \rangle \cap S = \langle a_S, b_S \rangle$, where $a_S, b_S \in \Ends(S)$. Thus, we have
$$\card\left(\{g \in G \mid g \cdot \base \in \langle a, b \rangle\} \cap B_G(k)\right)$$
$$\leqslant \sum_{S \in {\mathcal T}_k} \card\left(\{g \in G \mid g \cdot \base \in \langle a_S, b_S 
\rangle\} \cap B_G(k)\right)$$
$$\leqslant (2 M k + 1) (C(m-1) k^{m-1}) \leqslant C(T_{\base,m}) k^m,$$
where $C(T_{\base,m})$ is a constant which depends on $T_{\base,m}$ and $Z$, but not on $a$ and $b$.

\item[(m.2)] Assume that $T \neq T_{\base,m}$, that is, $\base \notin T$.

If 
$$g_0 \in \{g \in G \mid g \cdot \base \in \langle a, b \rangle\} \cap B_G(k),$$ 
then using the same argument as in (1.2) we can show that
$$\card\left(\{g \in G \mid g \cdot \base \in \langle a, b \rangle\} \cap B_G(k)\right)$$
$$\leqslant \card\left(\{g \in G \mid g \cdot \base \in \langle g_0^{-1} \cdot a, g_0^{-1} \cdot b \rangle\} \cap B_G(2k)\right).$$
Then from (m.1) we obtain
$$\card\left(\{g \in G \mid g \cdot \base \in \langle g_0^{-1} \cdot a, g_0^{-1} \cdot b \rangle\} \cap B_G(2k)\right) \leqslant C(T_{\base,m}) \cdot (2k)^m.$$
That is, $C(T) = 2^m C(T_{\base,m})$.

\end{enumerate}

By setting $C(m) \colon = 2^m C(T_{\base,m})$ we finish the induction step and {\bf Claim 2} is proved.

\end{proof}

\end{document}